\documentclass[a4paper, 12pt, reqno]{amsart}\usepackage[top=48truemm,bottom=43truemm,left=29.5truemm,right=29.5truemm]{geometry}

\usepackage{amscd, amsmath, amssymb, amsthm, comment}

\numberwithin{equation}{section}

\theoremstyle{plain}
\newtheorem{thm}{Theorem}[section]
\newtheorem{prop}[thm]{Proposition}
\newtheorem{cor}[thm]{Corollary}
\newtheorem{lem}[thm]{Lemma}

\newtheorem*{thms}{Theorem}

\newtheorem*{thmI}{Theorem I}
\newtheorem*{thmII}{Theorem II}

\theoremstyle{definition}
\newtheorem{defi}[thm]{Definition}
\newtheorem{example}[thm]{Example}
\newtheorem{rem}[thm]{Remark}
\newtheorem*{rem*}{Remark}

\begin{document}

\title[Orbits of path group actions induced by Hermann actions]{Curvatures and austere property of orbits of path group actions  induced by Hermann actions}
\author[M. Morimoto]{Masahiro Morimoto}

\address[M. Morimoto]{Osaka City University Advanced Mathematical Institute. 3-3-138 Sugimoto, Sumiyoshi-ku, Osaka, 558-8585, Japan}

\makeatletter
\@namedef{subjclassname@2020}{%
  \textup{2020} Mathematics Subject Classification}
\makeatother



\thanks{The author was partly supported by the Grant-in-Aid for Research Activity Start-up (No.\ 20K22309) and by 
Osaka City University Advanced Mathematical Institute (MEXT Joint Usage/Research Center on Mathematics and Theoretical Physics JPMXP0619217849).}

\maketitle

%%%%%
\begin{comment}
\vspace{-3mm}
\begin{center}\footnotesize
Osaka City University Advanced Mathematical Institute, 
\\
3-3-138 Sugimoto, Sumiyoshi-ku, Osaka, 558-8585, Japan
\end{center}

\begin{center} \footnotesize
\emph{E-mail adress}: mmasahiro0408@gmail.com
\end{center}

\vspace{-2mm}

\end{comment}
%%%%%

\begin{abstract}
It is known that an isometric action of a Lie group on a compact symmetric space gives rise to a proper Fredholm action of a path group on a path space via the gauge transformations. In this paper, supposing that the isometric action is a Hermann action (i.e. an isometric action of a symmetric subgroup of the isometry group) we give an explicit formula for the principal curvatures of orbits of the path group action and study the condition for those orbits to be austere, that is, the set of principal curvatures in the direction of each normal vector is invariant under the multiplication by minus one. To prove the results we essentially use the facts that Hermann actions are hyperpolar and all orbits of Hermann actions are curvature-adapted submanifolds. The results greatly extend the author's previous result in the case of the standard sphere and show that there exist a larger number of infinite dimensional austere submanifolds in Hilbert spaces.
\end{abstract}

\begin{comment}

\vspace{3mm}
\begin{center}\footnotesize
\emph{2020 Mathematics Subject Classification}: 53C40
\end{center}

\begin{center}\footnotesize
\emph{Keywords}: 
   principal curvature, 
   austere submanifold,  
\\
   Hermann action,
   hyperpolar action,
   proper Fredholm action
\end{center}

\bigskip

\end{comment}

%\clearpage

%%%%%%%
\section{Introduction} 
One way to study submanifolds in a compact symmetric space $G/K$ is to consider their lifts into a certain Hilbert space fibered over $G/K$. Let $V_\mathfrak{g} := L^2([0,1], \mathfrak{g})$ denote the Hilbert space of all $L^2$-paths from $[0,1]$ to the Lie algebra $\mathfrak{g}$ of $G$. Terng and Thorbergsson \cite{TT95} introduced a natural Riemannian submersion $\Phi: V_\mathfrak{g} \rightarrow G $ which is called the \emph{parallel transport map}. Denote by $\pi : G \rightarrow G/K$ the natural Riemannian submersion and consider the composition $\pi \circ \Phi : V_\mathfrak{g} \rightarrow G \rightarrow G/K$ which is a Riemannian submersion denoted by $\Phi_K$. It follows that if $N$ is a closed submanifold of $G/K$ then the inverse image $\Phi_K^{-1}(N)$ is a \emph{proper Fredholm} (PF) submanifold of $V_\mathfrak{g}$ (\cite{Ter89}). By definition the shape operators of PF submanifolds are compact self-adjoint operators. Although $\Phi_K^{-1}(N)$ is infinite dimensional, many techniques and results in the finite dimensional Euclidean case are still valid in the case of Hilbert space $V_\mathfrak{g}$. Applying those techniques and results to the lifted submanifolds they studied submanifold geometry in symmetric spaces. It is a fundamental problem to show the geometrical relation between $N$ and $\Phi_K^{-1}(N)$. 

Afterwards Koike \cite{Koi02} gave a formula for the principal curvatures of $\Phi_K^{-1}(N)$ under the assumption that $N$ is \emph{curvature-adapted}, that is, for each normal vector $v$ at each $p \in N$ the Jacobi operator $R_v$ leaves the tangent space $T_p N$ invariant and the restriction $R_v|_{T_p N}$ commutes with the shape operator $A_v$.  The author \cite{M2} corrected inaccuracies in that formula and studied the relation between two conditions:
\begin{enumerate}
\item[(A)] $N$ is an austere submanifold of $G/K$,
\item[(B)] $\Phi_K^{-1}(N)$ is an austere PF submanifold of $V_\mathfrak{g}$.
\end{enumerate}
Here a submanifold is called \emph{austere} (\cite{HL82}) if the set of principal curvatures with multiplicities in the direction of each normal vector is invariant under the multiplication by $(-1)$. The author showed (\cite[Theorem 4.1]{M2}):
\begin{thms}[\cite{M2}]
If $G/K$ is the standard sphere then \textup{(A)} and \textup{(B)} are equivalent.
\end{thms}

The purpose of this paper is to extend this theorem to the case that $G/K$ is not necessarily the standard sphere. However there are two difficulties to do this. The first one is that $N$ must be curvature-adapted in order to use the formula for the principal curvatures of $\Phi_K^{-1}(N)$, otherwise there is no way to compute those curvatures. The second one is that even if $N$ is curvature-adapted the principal curvatures of $\Phi_K^{-1}(N)$ and their multiplicities are complicated in general and it is not clear whether the austere properties of $N$ and $\Phi_K^{-1}(N)$ are equivalent or not.

In this paper we let $G/K$ be a symmetric space of compact type and suppose that $N$ is an orbit of a \emph{Hermann action} (\cite{Her60}), that is, an isometric action of a symmetric subgroup $H$ of $G$ on $G/K$. Here a closed subgroup $H$ of $G$ is called \emph{symmetric} if there exists an involutive automorphism $\tau$ of $G$ such that $H$ lies between the fixed point subgroup $G^\tau$ and its identity component. We know that all orbits of Hermann actions are curvature-adapted submanifolds (\cite{GT07}). Moreover we can explicitly describe the principal curvatures of orbits of Hermann actions (\cite{Ohno21}). Furthermore any Hermann action is \emph{hyperpolar} (\cite{Her62}, \cite{HPTT95}), that is, there exists a closed connected totally geodesic submanifold $\Sigma$ of $G/K$ which is flat in the induced metric and meets every orbit orthogonally. From this property we only have to consider normal vectors which are tangent to a fixed $\Sigma$ when studying the austere property of orbits (see Lemma \ref{lem2}).

We remark that if $N$ is an orbit of the $H$-action then the inverse image $\Phi_K^{-1}(N)$ is an orbit of a path group action. Let $\mathcal{G} := H^1([0, 1], G)$ denote the Hilbert Lie group of all Sobolev $H^1$-paths from $[0,1]$ to $G$. Consider the isometric action of $\mathcal{G}$ on $V_\mathfrak{g}$ given by the gauge transformations
\begin{equation*}
g*u = gug^{-1} - g' g^{-1},
\end{equation*}
where $g \in \mathcal{G}$, $u \in V_\mathfrak{g}$ and $g'$ denotes the weak derivative of $g$. Then the subgroup 
\begin{equation*}
P(G, H \times K)
=
\{g \in \mathcal{G} \mid g(0) \in H, \ g(1) \in K\}
\end{equation*}
acts on $V_\mathfrak{g}$ by the same formula. It follows that if $N$ is the $H$-orbit through $(\exp w)K$ for $w \in \mathfrak{g}$ then $\Phi_K^{-1}(N)$ is just the $P(G, H \times K)$-orbit through the constant path $\hat{w}$ with value $w$ (\cite{Ter95}). Therefore the conditions (A) and (B) are restated as follows:
\begin{enumerate}
\item[(A)] the orbit $H \cdot (\exp w) K$ through $(\exp w) K$ is an austere submanifold of $G/K$,
\item[(B)] the orbit $P(G, H \times K)* \hat{w}$ through $\hat{w}$ is an austere PF submanifold of $V_\mathfrak{g}$.
\end{enumerate}
Note that since the $H$-action is hyperpolar the $P(G, H \times K)$-action is also hyperpolar (\cite{HPTT95}, \cite{Ter95}).

In this paper we first derive an explicit formula for the principal curvatures of $P(G, H \times K)$-orbits (Theorem \ref{thm1}), which unifies and generalizes some results by Terng \cite{Ter89}, Pinkall-Thorbergsson \cite{PiTh90} and Koike \cite{Koi11} (see Remark \ref{rem5.4}). Then using this explicit formula we study the relation between (A) and (B). To explain the results we write $\sigma$ and $\tau$ for the involutions of $G$ associated to the symmetric subgroups $K$ and $H$ respectively. We denote by $\mathfrak{g} = \mathfrak{k} + \mathfrak{m}$ (resp.\ $\mathfrak{g} = \mathfrak{h} + \mathfrak{p}$) the decomposition into the $(\pm1)$-eigenspaces of the differential of $\sigma$ (resp.\ $\tau$). Take a maximal abelian subspace $\mathfrak{t}$ in $\mathfrak{m} \cap \mathfrak{p}$ and denote by $\Delta$ the root system of $\mathfrak{t}$ associated to the adjoint representation of $\mathfrak{t}$ on $\mathfrak{g}^\mathbb{C}$. We will prove the following theorem (Theorem \ref{main1}):
\begin{thmI}
If $\Delta$ is a reduced root system then \textup{(A)} and \textup{(B)} are equivalent. 
\end{thmI}

Then without supposing that $\Delta$ is a reduced root system we will prove the following theorem (Theorems \ref{prop1}, \ref{thm:commute} and \ref{thm:main}): 

\begin{thmII}\ 
\begin{enumerate}
\item Suppose that $\sigma = \tau$. Then \textup{(A)} and \textup{(B)} are equivalent.
\item Suppose that $\sigma$ and $\tau$ commute. Then \textup{(A)} implies \textup{(B)}.
\item Suppose that $G$ is simple. Then \textup{(A)} implies \textup{(B)}. 
\end{enumerate}
\end{thmII}

Note that (B) does not imply (A) in the cases (ii) and (iii). In fact we will show a counterexample of a minimal $H$-orbit which is \emph{not} austere but the corresponding $P(G, H \times K)$-orbit is austere (cf.\ Section \ref{counterexample}). Without the assumption of (ii) or (iii) we do not know whether (A) implies (B) or not, because in the non-simple case there exist many non-commutative pairs of involutive automorphisms of $G$ (\cite{Mat02}). However the above theorems greatly extend the previous theorem in the spherical case and cover all known examples of austere orbits of Hermann actions (\cite{Ika11}, \cite{Ohno21}). Applying those examples to the above theorems we obtain a larger number of infinite dimensional austere PF submanifolds in Hilbert spaces.  Notice that so obtained  austere PF submanifolds are not totally geodesic due to \cite{M1}. 

This paper is organized as follows. In Section \ref{preliminaries} we review basic knowledge on $P(G, H \times K)$-actions and the parallel transport map. In Section \ref{Hermann} we review fundamental results on the submanifold geometry of orbits of Hermann actions. In Section \ref{cap} we introduce a hierarchy of curvature-adapted submanifolds in symmetric spaces and formulate the curvature-adapted property of orbits of Hermann actions. In Section \ref{pcvptm} we refine the formula for the principal curvatures of $\Phi_K^{-1}(N)$ (\cite{Koi02},  \cite{M2}) so that orbits of Hermann actions can be applied. In Section \ref{pcpo}, applying orbits of Hermann actions to the refined formula we derive an explicit formula for the principal curvatures of $P(G, H \times K)$-orbits. In Section \ref{austere:reduced}, using this explicit formula we formulate the conditions (A) and (B) in terms of roots in $\Delta$ and prove Theorem I. In Section \ref{austere:general} we show an inequality between the multiplicities of roots $\alpha$ and $2 \alpha$ in $\Delta$ and prove Theorem II. Finally in Section \ref{counterexample} we show a counterexample to the converse of (ii) and (iii) of Theorem II and mention further remarks on the converse. 

%%%%%%%
\section{Preliminaries}\label{preliminaries}

Let $G$ be a connected compact semisimple Lie group and $K$ a closed subgroup of $G$. Suppose that $K$ is a symmetric subgroup of $G$, that is, there exists an involutive automorphism $\sigma$ of $G$ satisfying the condition 
$G^\sigma_0 \subset K \subset G^\sigma$, where $G^\sigma$ denotes the fixed point subgroup of $G$ and $G^\sigma_0$ its identity component.  We denote by $\mathfrak{g}$ and $\mathfrak{k}$ the Lie algebras of $G$ and $K$ respectively. The differential of $\sigma$ induces an involutive automorphism of $\mathfrak{g}$, which is still denoted by $\sigma$. The direct sum decomposition $\mathfrak{g} = \mathfrak{k} + \mathfrak{m}$ into the $(\pm1)$-eigenspaces of $\sigma$ is called the \emph{canonical decomposition}. We fix an $\operatorname{Ad}(G)$-invariant inner product $\langle \cdot, \cdot \rangle$ of $\mathfrak{g}$ which is a negative multiple of the Killing form of $\mathfrak{g}$. Then it is invariant under all automorphisms of $\mathfrak{g}$ and the canonical decomposition is orthogonal. We equip the corresponding bi-invariant Riemannian metric with $G$ and the $G$-invariant Riemannian metric with the homogeneous space $G/K$.  Then $M := G/K$ is a symmetric space of compact type and the projection $\pi: G \rightarrow M$ is a Riemannian submersion with totally geodesic fiber. 

We denote by $\mathcal{G} := H^1([0,1], G)$ the path group of all Sobolev $H^1$-paths from $[0,1]$ to $G$ and by $V_\mathfrak{g} := L^2([0,1], \mathfrak{g})$ the path space of all $L^2$-paths from $[0,1]$ to $\mathfrak{g}$. Then $\mathcal{G}$ is a Hilbert Lie group and $V_\mathfrak{g}$ a separable Hilbert space. We consider the isometric action of $\mathcal{G}$ on $V_\mathfrak{g}$ given by the gauge transformations
\begin{equation*}
g * u := gug^{-1} - g' g^{-1},
\end{equation*}
where $g \in \mathcal{G}$ and $u \in V_\mathfrak{g}$. We know that this action is proper and Fredholm (\cite[Theorem 5.8.1]{PT88}, \cite[Section 4]{Ter89}). For any closed subgroup $L$ of $G \times G$ the subgroup
\begin{equation*}
P(G, L) := \{g \in \mathcal{G} \mid (g(0), g(1)) \in L\}
\end{equation*}
acts on $V_\mathfrak{g}$ by the same formula. It follows that the $P(G, L)$-action is also proper and Fredholm (\cite[p.\ 132]{Ter95}). Thus every orbit of the $P(G, L)$-action is a proper Fredholm (PF) submanifold of $V_\mathfrak{g}$ (\cite[Theorem 7.1.6]{PT88}). We know that the $P(G, \{e\} \times G)$-action on $V_\mathfrak{g}$ is simply transitive (\cite[Corollary 4.2]{TT95}) and that the $P(G, G \times \{e\})$-action on $V_\mathfrak{g}$ is also simply transitive (\cite[Section 5]{M1}). 

The \emph{parallel transport map} (\cite{Ter95}, \cite{TT95}) $\Phi: V_\mathfrak{g} \rightarrow G$ is defined by 
\begin{equation*}
\Phi(u) := g_u(1),
\end{equation*}
where $g_u \in \mathcal{G}$ is the unique solution to the linear ordinary differential equation
\begin{equation*}
\left\{\begin{array}{l}
g_u^{-1} g'_u = u,
\\
g_u(0) = e.
\end{array}\right.
\end{equation*}
We know that $\Phi$ is a Riemannian submersion and a principal $\Omega_e(G)$-bundle, where $\Omega_e(G) = P(G, \{e\} \times \{e\})$ denotes the based loop group (\cite[Corollary 4.4, Theorem 4.5]{TT95}). The normal space of the fiber $\Phi^{-1}(e)$ at $\hat{0} \in V_\mathfrak{g}$ is identified with the subspace $\hat{\mathfrak{g}} = \{\hat{x} \mid x \in \mathfrak{g}\}$ of $V_\mathfrak{g}$, where $\hat{x}$ denotes the constant path with value $x$. It follows that $\Phi(\hat{x}) = \exp x$. The composition
\begin{equation*}
\Phi_K := \pi \circ \Phi : V_\mathfrak{g} \rightarrow G \rightarrow M
\end{equation*} 
is a Riemannian submersion which is also called the parallel transport map. 

We consider the isometric action of $G$ on $M$ defined by 
\begin{equation*}
b \cdot (aK) := (ba) K,
\end{equation*}
where $b \in G$ and $aK \in M$, and the isometric action of $G \times G$ on $G$ defined by 
\begin{equation*}
(b, c) \cdot a := b a c^{-1},
\end{equation*}
where $(b,c) \in G \times G$ and $a \in G$. Then $\pi$ and $\Phi$ have the following equivariant properties (\cite[Proposition 1.1 (i)]{Ter95}):
\begin{align}
&\label{equiv1}
\pi((b,c) \cdot a) = b \cdot \pi(a) \qquad \text{for} \ (b,c) \in G \times K \ \text{and}\ a \in G,
\\
& \label{equiv2}
\Phi(g * u) = (g(0), g(1)) \cdot \Phi(u)  \qquad \text{for} \ g \in \mathcal{G} \ \text{and}\  u \in V_\mathfrak{g}.
\end{align}
From these we have
\begin{equation}\label{equiv}
\Phi_K(g * u) = g(0) \Phi_K(u)
\qquad \text{for} \ g \in P(G, G \times K) \ \text{and}\  u \in V_\mathfrak{g}.
\end{equation}

Let $H$ be a closed subgroup of $G$; in later sections we will suppose that it is a symmetric subgroup of $G$. We denote by $\mathfrak{h}$ the Lie algebra of $H$ and $\mathfrak{g} = \mathfrak{h} + \mathfrak{p}$ the orthogonal direct sum decomposition. Then $H$ acts on $M$, the subgroup $H \times K$ acts on $G$ and the subgroup $P(G, H \times K)$ acts on $V_\mathfrak{g}$. We know the following relations for orbits (\cite[Proposition 1.1 (ii)]{Ter95}):
\begin{equation*}
(H \times K) \cdot a = \pi^{-1}(H \cdot aK)
 \quad \text{and} \quad
P(G, H \times K) * u = \Phi^{-1}((H \times K) \cdot \Phi(u)).
\end{equation*}
Thus we have 
\begin{equation}\label{inverse}
P(G, H \times K) * u
=
\Phi_K^{-1}(H \cdot \Phi_K(u)).
\end{equation}
Then we obtain the commutative diagram
\begin{equation*}
\begin{array}{ccccccccc}
\mathcal{G} \!&\! \supset \!&\! P(G, H \times K) \!&\!\curvearrowright\!&\! V_\mathfrak{g} \!&\! \supset \!&\! P(G, H \times K) * u  \!&\!=\!&\! \Phi^{-1}((H \times K) \cdot a) \medskip
\\
\psi \downarrow \ \ \ \!&\! \!&\! \psi \downarrow \ \ \ \ \!&\! \!&\! \Phi \downarrow \ \ \ \!&\!\!&\!\Phi \downarrow \ \ \ \!&\!\!&\!\Phi \downarrow \ \ \ \medskip
\\
G \times G \!&\! \supset \!&\! H \times K \!&\!\curvearrowright\!&\! G \!&\! \supset \!&\!  (H \times K) \cdot a \!&\! = \!&\! \pi^{-1}(H \cdot a K) \medskip
\\
p \downarrow \ \ \ \!&\! \!&\! p \downarrow \ \ \ \!&\! \!&\! \pi \downarrow \ \ \ \!&\!\!&\!\pi \downarrow \ \ \ \!&\!\!&\!\medskip
\\
G\!&\! \supset \!&\! H\!&\!\curvearrowright\!&\! M \!&\! \supset \!&\! H \cdot aK \!&\!\!&\! (\Phi(u) = a),
\end{array}
\end{equation*}
where $p$ denotes the projection onto the first component and $\psi$ the submersion defined by $\psi(g) := (g(0),g(1))$ for $g \in \mathcal{G}$. We know that the following conditions are equivalent: the orbit $H \cdot aK$ is a minimal submanifold of $M$, the orbit $(H \times K) \cdot a$ is a minimal submanifold of $G$, and the orbit $P(G, H \times K) * u$ through $u \in \Phi^{-1}(a)$ is a minimal PF submanifold of $V_\mathfrak{g}$ (\cite[Theorem, 4.12]{KT93}, \cite[Lemma 5.2]{HLO06}).

Recall that an isometric action of a compact Lie group $A$ on a Riemannian manifold $X$ is called \emph{polar} if there exists a closed connected submanifold $\Sigma$ of $X$ which meets every $A$-orbit and is orthogonal to the $A$-orbits at every point of intersection. Such a $\Sigma$ is called a \emph{section}, which is automatically totally geodesic in $X$. If $\Sigma$ is flat in the induced metric then the action is called \emph{hyperpolar} (\cite{HPTT95}). For a proper Fredholm action on a Hilbert space we can define it to be hyperpolar by the similar way. We know that the following conditions are equivalent (\cite[Proposition 2.11]{HPTT95}, \cite[Theorem 1.2]{Ter95}, \cite[Lemma 4]{GT02}):
\begin{enumerate}
\item the $H$-action on $M$ is hyperpolar, 
\item the $H \times K$-action on $G$ is hyperpolar,
\item the $P(G,H \times K)$-action on $V_\mathfrak{g}$ is hyperpolar.
\end{enumerate}
Since we fixed a bi-invariant Riemannian metric on $G$ induced by a negative multiple of the Killing form of $\mathfrak{g}$, the condition (ii) is equivalent to the existence of a $c$-dimensional abelian subspace $\mathfrak{t}$ in $\mathfrak{m} \cap \mathfrak{p}$ where $c$ is the cohomogeneity of the $H \times K$-action (\cite[Theorem 2.1]{HPTT95}). Then $\pi(\exp \mathfrak{t})$, $\exp \mathfrak{t}$ and $\hat{\mathfrak{t}} = \{\hat{x} \mid x \in \mathfrak{t}\}$ are sections of the $H$-action, the $H \times K$-action and the $P(G, H \times K)$-action respectively. If the actions are hyperpolar then the following conditions are equivalent (\cite[Theorem 1.2]{Ter95}): $aK \in M$ is a regular point of the $H$-action, $a \in G$ is a regular point of the $H \times K$-action, and $u \in \Phi^{-1}(a)$ is a regular point of the $P(G, H \times K)$-action. Here a point is called \emph{regular} if the orbit though it is principal.

% In general it is not easy to compute the principal curvatures of orbits of those actions. However if $H$ is a symmetric subgroup of $G$, that is, the $H$-action is a Hermann action, then it is possible to describe the principal curvatures of those orbits Lie algebraically via the root space decompositions. In the rest of this paper we will focus on the case of Hermann actions.

%%%%%%%
\section{Submanifold geometry of orbits of Hermann actions}\label{Hermann}

In this section we review fundamental results on the submanifold geometry of orbits of Hermann actions. For details, see 
Ohno \cite{Ohno21} (see also Goertsches-Thorbergsson \cite{GT07} and Ikawa \cite{Ika11}). Throughout this section $M = G/K$ denotes a symmetric space of compact type and $H$ a symmetric subgroup of $G$. We denote by $\sigma$ and $\tau$ the involutions of $G$ associated to $K$ and $H$ respectively and by $\mathfrak{g} = \mathfrak{k} + \mathfrak{m}$ and $\mathfrak{g} = \mathfrak{h} + \mathfrak{p}$ the canonical decompositions. We choose and fix a maximal abelian subspace $\mathfrak{t}$ in $\mathfrak{m} \cap \mathfrak{p}$ so that $\Sigma := \pi(\exp \mathfrak{t})$ is a section of the Hermann action. 

Consider the root space decomposition of $\mathfrak{g}^\mathbb{C}$ with respect to $\mathfrak{t}$:
\begin{equation*}
\mathfrak{g}^\mathbb{C} 
= 
\mathfrak{g}(0) + \sum_{\alpha \in \Delta} \mathfrak{g}(\alpha),
\end{equation*}
\begin{align*}
\mathfrak{g}(0)
&=
\{z \in \mathfrak{g}^\mathbb{C} \mid \operatorname{ad}(\eta) z = 0 \ \text{ for all} \ \eta \in \mathfrak{t} \},
\\
\mathfrak{g}(\alpha)
&=
\{z \in \mathfrak{g}^\mathbb{C} \mid \operatorname{ad}(\eta) z = \sqrt{-1} \langle \alpha, \eta \rangle z \ \text{ for all} \ \eta \in \mathfrak{t}\},
\end{align*}
where $\Delta = \{\alpha \in \mathfrak{t} \backslash \{0\} \mid \mathfrak{g}(\alpha) \neq \{0\}\}$ is a root system of $\mathfrak{t}$ (\cite[Lemma 4.12]{Ika11}). Since $\overline{\mathfrak{g}(\alpha)} = \mathfrak{g}(- \alpha)$, where the bar denotes the complex conjugation, the real form is 
\begin{equation*}
\mathfrak{g} = \mathfrak{g}_0 + \sum_{\alpha \in \Delta^+} \mathfrak{g}_\alpha,
\end{equation*}
\begin{equation*}
\mathfrak{g}_0 = \mathfrak{g}(0) \cap \mathfrak{g}
, \qquad
\mathfrak{g}_\alpha = (\mathfrak{g}(\alpha) + \mathfrak{g}(- \alpha)) \cap \mathfrak{g}.
\end{equation*}
Note that
\begin{align*}
\mathfrak{g}_0
&=
\{x \in \mathfrak{g} \mid \operatorname{ad}(\eta) x = 0 \ \text{ for all} \ \eta \in \mathfrak{t}  \},
\\
\mathfrak{g}_\alpha
&=
\{x \in \mathfrak{g} \mid  \operatorname{ad}(\eta)^2 x = - \langle \alpha, \eta \rangle ^2 x \ \text{ for all} \ \eta \in \mathfrak{t} \}.
\end{align*}
Since $\sigma$ commutes with $\operatorname{ad}(\eta)^2$ for all $\eta \in \mathfrak{t}$  we have 
\begin{equation*}
\mathfrak{k} = \mathfrak{k}_0 + \sum_{\alpha \in \Delta^+} \mathfrak{k}_\alpha
, \qquad
\mathfrak{m} = \mathfrak{m}_0 + \sum_{\alpha \in \Delta^+} \mathfrak{m}_\alpha.
\end{equation*}
\begin{equation*}
\begin{array}{lcl}
\mathfrak{k}_0
=
\mathfrak{g}_0 \cap \mathfrak{k},
&&
\mathfrak{m}_0
=
\mathfrak{g}_0 \cap \mathfrak{m},
\medskip
\\
\mathfrak{k}_\alpha
=
\mathfrak{g}_\alpha \cap \mathfrak{k},
&&
\mathfrak{m}_\alpha
=
\mathfrak{g}_\alpha \cap \mathfrak{m}.
\end{array}
\end{equation*}
We define a linear orthogonal transformation $\psi_\alpha$ of $\mathfrak{g}_\alpha$ by 
\begin{equation}\label{psi}
\psi_\alpha (x) 
:= 
\frac{1}{\langle \alpha, \alpha \rangle} \operatorname{ad}(\alpha) x
\qquad \text{for} \ x \in \mathfrak{g}_{\alpha}.
\end{equation}
An equivalent definition is that 
\begin{equation*}
\psi_\alpha(z + \bar{z}) := \sqrt{-1}(z - \bar{z})
\qquad \text{for} \ z \in \mathfrak{g}(\alpha).
\end{equation*}
Since $\sigma \circ \psi_\alpha = -(\psi_\alpha \circ \sigma )$ we have a linear isometry $\psi_\alpha: \mathfrak{m}_\alpha \rightarrow \mathfrak{k}_\alpha$. Set 
\begin{equation*}
m(\alpha) 
:=
\dim \mathfrak{k}_\alpha = \dim \mathfrak{m}_\alpha.
\end{equation*}
By setting $x^\alpha_i := \psi_\alpha(y^\alpha_i)$ we can take bases $\{x^\alpha_i\}_{i = 1}^{m(\alpha)}$ of $\mathfrak{k}_{\alpha}$ and $\{y^{\alpha}_i\}_{i = 1}^{m(\alpha)}$ of $\mathfrak{m}_\alpha$ satisfying 
\begin{equation}
[\eta, x^\alpha_i] = - \langle \alpha, \eta \rangle y^\alpha_i
\quad \text{and} \quad
[\eta, y^\alpha_i] = \langle \alpha, \eta \rangle x^\alpha_i
\end{equation}
for any $\eta \in \mathfrak{t}$.

The root space decompositions above are refined by combining a decomposition derived from the involutions $\sigma$ and $\tau$. More precisely we consider the composition
\begin{equation*}
\sigma \circ \tau :\mathfrak{g} \rightarrow \mathfrak{g}
\end{equation*}
and the eigenspace decomposition
\begin{equation*}
\mathfrak{g}^\mathbb{C}
=
\sum_{\epsilon \in U(1)} \mathfrak{g}(\epsilon),
\end{equation*}
\begin{equation*}
\mathfrak{g}(\epsilon)
=
\{z \in \mathfrak{g}^\mathbb{C} \mid (\sigma \circ  \tau)(z) = \epsilon z\},
\end{equation*}
where the eigenvalues belong to $U(1) = \{\epsilon \in \mathbb{C} \mid |\epsilon| = 1\}$. For each $\epsilon \in U(1)$ we denote by $\arg \epsilon$ its argument satisfying $- \pi < \arg \epsilon \leq \pi$. Since $\sigma \circ \tau$ commutes with $\operatorname{ad}(\eta)$ for all $\eta \in \mathfrak{t}$ we have
\begin{equation*}
\mathfrak{g}^\mathbb{C}
=
\sum_{\epsilon \in U(1)} \mathfrak{g}(0, \epsilon)
+
\sum_{\alpha \in \Delta}
\sum_{\epsilon \in U(1)} 
\mathfrak{g}(\alpha, \epsilon),
\end{equation*}
\begin{equation*}
\mathfrak{g}(0, \epsilon) = \mathfrak{g}(0) \cap \mathfrak{g}(\epsilon)
, \qquad
\mathfrak{g}(\alpha, \epsilon) = \mathfrak{g}(\alpha) \cap \mathfrak{g}(\epsilon).
\end{equation*}
Since $\overline{\mathfrak{g}(\alpha, \epsilon)} = \mathfrak{g}(- \alpha, \epsilon^{-1})$ the real form is
\begin{equation*}
\mathfrak{g}
=
\sum_{\epsilon \in U(1)_{\geq 0}} \mathfrak{g}_{0, \epsilon}
+
\sum_{\alpha \in \Delta^+}
\sum_{\epsilon \in U(1)} \mathfrak{g}_{\alpha, \epsilon},
\end{equation*}
\begin{equation*}
U(1)_{\geq 0} = \{\epsilon \in U(1) \mid \operatorname{Im}(\epsilon) \geq 0\},
\end{equation*}
\begin{align*}
&
\mathfrak{g}_{0, \epsilon}
=
( \mathfrak{g}(0, \epsilon) + \mathfrak{g}(0, \epsilon^{-1}) \cap \mathfrak{g},
\\
&
\mathfrak{g}_{\alpha, \epsilon}
=
(\mathfrak{g}(\alpha, \epsilon) + \mathfrak{g}(- \alpha, \epsilon^{-1})) \cap \mathfrak{g}.
\end{align*}
Setting $\rho^+ = \sigma \circ \tau + \tau \circ \sigma$ and $\rho^- = \sigma \circ \tau - \tau \circ \sigma$ we can write
\begin{align*}
\mathfrak{g}_{0, \epsilon}
&=
\{x \in \mathfrak{g}_0 \mid  \rho^+(x) =  2 \operatorname{Re}(\epsilon) x \},
\\
\mathfrak{g}_{\alpha, \epsilon}
&=
\{x \in \mathfrak{g}_\alpha \mid   \rho^+(x) = 2 \operatorname{Re}(\epsilon) x, \ \rho^-(x) =  2 \operatorname{Im}(\epsilon) \psi_\alpha(x) \},
\end{align*}
where $\operatorname{Re}(\epsilon)$ and $\operatorname{Im}(\epsilon)$ denote the real and imaginary parts of $\epsilon$ respectively. Since $\mathfrak{g}_{0, \epsilon}$ and $\mathfrak{g}_{\alpha, \epsilon}$ are invariant under $\sigma$ we have
\begin{align*}
&
\mathfrak{k} 
= 
\sum_{\epsilon \in U(1)_{\geq 0}} \mathfrak{k}_{0, \epsilon} 
+
\sum_{\alpha \in \Delta^+} 
\sum_{\epsilon \in U(1)} \mathfrak{k}_{\alpha, \epsilon},
\\
&
\mathfrak{m} 
= 
\sum_{\epsilon \in U(1)_{\geq 0}} \mathfrak{m}_{0, \epsilon} 
+
\sum_{\alpha \in \Delta^+} 
\sum_{\epsilon \in U(1)} \mathfrak{m}_{\alpha, \epsilon},
\end{align*}
\begin{equation*}
\begin{array}{lcl}
\mathfrak{k}_{0, \epsilon}
=
\mathfrak{g}_{0, \epsilon} \cap \mathfrak{k},
& &
\mathfrak{k}_{\alpha, \epsilon}
=
\mathfrak{g}_{\alpha, \epsilon} \cap \mathfrak{k} ,
\medskip
\\
\mathfrak{m}_{0, \epsilon}
=
\mathfrak{g}_{0, \epsilon} \cap \mathfrak{m},
& &
\mathfrak{m}_{\alpha, \epsilon}
=
\mathfrak{g}_{\alpha, \epsilon} \cap \mathfrak{m}.
\end{array}
\end{equation*}
Since $\mathfrak{g}_{\alpha, \epsilon}$ is invariant under $\psi_\alpha$ we have a linear isometry $\psi_\alpha: \mathfrak{m}_{\alpha, \epsilon} \rightarrow \mathfrak{k}_{\alpha, \epsilon}$. Set 
\begin{equation*}
m(\alpha, \epsilon)
:=
\dim \mathfrak{k}_{\alpha, \epsilon}
=
\dim \mathfrak{m}_{\alpha, \epsilon}.
\end{equation*}
Then similarly we can take bases $\{x^{\alpha, \epsilon}_i\}_{i = 1}^{m(\alpha,\epsilon)}$ of $\mathfrak{k}_{\alpha, \epsilon}$ and $\{y^{\alpha, \epsilon}_i\}_{i = 1}^{m(\alpha, \epsilon)}$ of $\mathfrak{m}_{\alpha, \epsilon}$ satisfying
\begin{equation}\label{basis-relation}
[\eta, x^{\alpha, \epsilon}_i] = - \langle \alpha, \eta \rangle y^{\alpha, \epsilon}_i
\quad \text{and} \quad
[\eta, y^{\alpha, \epsilon}_i] =  \langle \alpha, \eta \rangle x^{\alpha, \epsilon}_i
\end{equation}
for any $\eta \in \mathfrak{t}$. 

We now take $w \in \mathfrak{t}$, set $a := \exp w$ and consider the orbit $N := H \cdot aK$ through $aK$. Denote by $L_a$ the isometry of $M$ defined by $L_a(bK) := (ab)K$. Identifying $T_{eK} M$ with $\mathfrak{m}$ we can describe the tangent space and the normal space of $N$ as follows (\cite[p.\ 12]{Ohno21}):
\begin{align}
\label{tangent2}
T_{aK}N
&=
dL_{a}( \ \ 
\sum_{\substack{\epsilon  \in U(1)_{\geq 0} \\ \epsilon \neq 1}} \mathfrak{m}_{0, \epsilon} 
+
\sum_{\alpha \in \Delta^+}
\sum_{\substack{\epsilon \in U(1) \\ \langle \alpha, w  \rangle + \frac{1}{2} \arg \epsilon \notin \pi\mathbb{Z}}}
\mathfrak{m}_{\alpha, \epsilon}
\ \  ),\\
\label{normal2}
T^\perp_{aK}N
&=
dL_{a}(\ 
\qquad \mathfrak{t}  \qquad 
+ \quad 
\sum_{\alpha \in \Delta^+}
\sum_{\substack{\epsilon \in U(1)\\ \langle \alpha, w  \rangle + \frac{1}{2} \arg \epsilon \in  \pi\mathbb{Z}}}\mathfrak{m}_{\alpha, \epsilon}
\ \  ).
\end{align}
Moreover the decomposition \eqref{tangent2} is just the eigenspace decomposition of the family shape operators $\{A^N_{dL_a(\xi)}\}_{\xi \in \mathfrak{t}}$. In fact (\cite[p.\ 17]{Ohno21}):
\begin{align*}
dL_a (\mathfrak{m}_{0, \epsilon}): &\ \ \text{the eigenspace associated with the eigenvalue $0$},
\\
dL_a(\mathfrak{m}_{\alpha, \epsilon}): &\ \ \text{the eigenspace associated with}
\\ & \ \   %%%
\text{the eigenvalue $\textstyle - \langle \alpha, \xi\rangle \cot ( \langle\alpha, w \rangle + \frac{1}{2} \arg \epsilon)$}
\end{align*}
for each $\xi \in \mathfrak{t}$. If $\sigma$ and $\tau$ commute then $\epsilon = \pm 1$ and thus we get (\cite[Theorem 5.3]{GT07}):
\begin{align}
\label{tangent3:commute}
T_{aK}N
&=
dL_{a}(\ \ 
\mathfrak{m}_{0} \cap \mathfrak{h} 
+
\sum_{\substack{\alpha \in \Delta^+ \\ \langle \alpha, w  \rangle \notin \pi\mathbb{Z}}}
\mathfrak{m}_{\alpha} \cap \mathfrak{p}
+
\sum_{\substack{\alpha \in \Delta^+ \\ \langle \alpha, w  \rangle + \pi/2  \notin  \pi\mathbb{Z}}}
\mathfrak{m}_{\alpha} \cap \mathfrak{h}
\ \  ),\\
\label{normal3:commute}
T^\perp_{aK}N
&=
dL_{a}(\ \ \ 
\quad \mathfrak{t} \ \quad 
+
\sum_{\substack{\alpha \in \Delta^+ \\ \langle \alpha, w  \rangle \in \pi\mathbb{Z}}}
\mathfrak{m}_{\alpha} \cap \mathfrak{p}
+
\sum_{\substack{\alpha \in \Delta^+ \\ \langle \alpha, w  \rangle + \pi/2  \in \pi\mathbb{Z}}}
\mathfrak{m}_{\alpha} \cap \mathfrak{h}
\ \  ),
\end{align}
\begin{align*}
dL_a(\mathfrak{m}_{0} \cap \mathfrak{h}): & \ \ \text{the eigenspace associated with the eigenvalue $0$},
\\
dL_a(\mathfrak{m}_\alpha \cap \mathfrak{p}) :& \  \ \text{the eigenspace associated with the eigenvalue $ - \langle \alpha, \xi\rangle \cot \langle\alpha, w \rangle$},
\\
dL_a(\mathfrak{m}_\alpha \cap \mathfrak{h}): & \   \ \text{the eigenspace associated with the eigenvalue $ \langle \alpha, \xi\rangle \tan \langle\alpha, w \rangle$}.
\end{align*}
In particular if $\sigma = \tau$ then we have (\cite[p.\ 122]{Tas85})
\begin{align}
\label{tangent4}
T_{aK}N
&=
dL_{a}(\ \ 
\sum_{\substack{\alpha \in \Delta^+ \\ \langle \alpha, w  \rangle \notin \pi\mathbb{Z}}}
\mathfrak{m}_{\alpha} \ \  ),
\\
\label{normal4}
T^\perp_{aK}N
&=
dL_{a}(\ \ 
\quad \mathfrak{t} \quad 
+
\sum_{\substack{\alpha \in \Delta^+ \\ \langle \alpha, w  \rangle \in \pi\mathbb{Z}}}
\mathfrak{m}_{\alpha}\ \  ),
\end{align}
\begin{equation*}
dL_a(\mathfrak{m}_\alpha) : \  \text{the eigenspace associated with the eigenvalue $ - \langle \alpha, \xi\rangle \cot \langle\alpha, w \rangle$}.
\end{equation*}

%%%%%%%
\section{The curvature-adapted property}\label{cap}

In this section we formulate the curvature-adapted property of orbits of Hermann actions. 

First we recall the concept of curvature-adapted submanifolds (\cite{BV92}). Let $N$ be a submanifold of a Riemannian manifold $M$. For each $v \in T^\perp_{p} N$ at each $p \in N$ the Jacobi operator $R_v$ is a symmetric linear transformation of $T_{p} M$ defined by 
\begin{equation*}
R_v(x) = R^M(x, v) v \qquad \text{for} \ x \in T_{p} M,
\end{equation*}
where $R^M$ denotes the curvature tensor of $M$. Then $N$ is called \emph{curvature-adapted} if for every $v \in T^\perp_p N$ at each $p \in N$ the Jacobi operator $R_v$ leaves $T_p N$ invariant and the restriction $R_v|_{T_p N}$ commutes with the shape operator $A_v^N$ of $N$. 

We now make the following definition:
\begin{defi}\label{ccad}
Let $M = G/K$ be a symmetric space of compact type and $N$ a submanifold of $M$. For an integer $c$ satisfying $1 \leq c \leq \operatorname{codim} N$ we say that $N$ is \emph{$c$-curvature-adapted} if for each $aK \in N$ the following two conditions are satisfied: 
\begin{enumerate}
\item for every $v \in T^\perp_{aK} N$ the Jacobi operator $R_v$ leaves $T_{aK} N$ invariant,

\item for each $v \in T^\perp_{aK} N$ there exists a $c$-dimensional abelian subspace  $\mathfrak{t}$ in $\mathfrak{m}$ satisfying $v \in dL_a(\mathfrak{t}) \subset T^\perp_{aK} N $ such that the union
\begin{equation*}
\{R_{dL_a(\xi)}|_{T_{aK} N}\}_{\xi \in \mathfrak{t}} \cup \{A^N_{dL_a(\xi)}\}_{\xi \in \mathfrak{t}}
\end{equation*}
is a commuting family of endomorphisms of $T_{aK} N$. 

\end{enumerate}
\end{defi}

Note that if $c = 1$ then $1$-curvature-adapted submanifolds are just curvature-adapted submanifolds in the original sense. 
Note also that if $aK = eK$ then $R_v$ is identified with $- \operatorname{ad}(v)^2$ since $M$ is a symmetric space. Typical examples of $c$-curvature-adapted submanifolds are given by the following proposition, which was essentially shown by Goertsches and Thorbergsson \cite[Corollaries 3.3 and 3.4]{GT07}:   
\begin{prop}[Goertsches-Thorbergsson \cite{GT07}]\label{GT}
All orbits of Hermann actions of cohomogeneity $c$ are $c$-curvature-adapted submanifolds. 
\end{prop}
\begin{proof}
Let $N$ be an orbit of a Hermann action $H \curvearrowright M$ of cohomogeneity $c$. Take $aK \in N$. Since $L_a^{-1} N = (a^{-1}Ka) \cdot eK$ we can assume $a K = eK$ without loss of generality. Take $v \in T^\perp_{eK} N$. Choose a maximal abelian subspace $\mathfrak{t}$ in $\mathfrak{m} \cap \mathfrak{p} = T^\perp_{eK} N$ containing $v$. Since $\pi(\exp \mathfrak{t})$ is a section of the Hermann action we have $\dim \mathfrak{t} = c$. Then it follows from the decomposition \eqref{tangent2} that the Jacobi operator $R_v = - \operatorname{ad}(v)^2$ leaves $T_{eK} N$ invariant and 
\begin{equation*}
R_v \circ R_w =R_w \circ R_v
, \quad
R_v|_{T_{eK} N} \circ A^N_w = A^N_w \circ R_v|_{T_{eK} N} 
, \quad
A^N_v \circ A^N_w = A^N_w \circ A^N_v
\end{equation*}
hold for any $v, w \in \mathfrak{t}$. Thus $N$ is a $c$-curvature-adapted submanifolds of $M$.
\end{proof}

\begin{rem}
We do not know whether all orbits of hyperpolar actions of cohomogeneity $c$ are $c$-curvature-adapted submanifolds or not. We know that any indecomposable hyperpolar action of cohomogeneity at least two on $M$ is orbit equivalent to a Hermann action (\cite{Kol17}). We also know that any cohomogeneity one action on $M$ is automatically hyperpolar (\cite[Corollary 2.13]{HPTT95}). There exist examples of cohomogeneity one actions on the standard sphere which are different from Hermann actions (\cite{HPTT95}, \cite{Kol02}). Since the standard sphere is of constant sectional curvature, all orbits of such cohomogeneity one actions are $1$-curvature-adapted submanifolds.
\end{rem}

Let $N$ be a $c$-curvature-adapted submanifold of a symmetric space $M = G/K$ of compact type. 
Take $aK \in N$. Choose and fix an arbitrary $c$-dimensional abelian subspace  $\mathfrak{t}$ in $\mathfrak{m}$ satisfying the condition (ii) of Definition \ref{ccad} for some $v \in T^\perp_{aK} N$. Consider the root space decomposition
\begin{equation*}
\mathfrak{k} = \mathfrak{k}_0 + \sum_{\alpha \in \Delta^+} \mathfrak{k}_\alpha
, \qquad
\mathfrak{m} = \mathfrak{m}_0 + \sum_{\alpha \in \Delta^+} \mathfrak{m}_\alpha,
\end{equation*}
\begin{align*}
\mathfrak{k}_0
&=
\{x \in \mathfrak{k} \mid \operatorname{ad}(\eta) x= 0 \ \ \text{for all } \eta \in \mathfrak{t} \},
\\
\mathfrak{k}_\alpha
&=
\{x \in \mathfrak{k} \mid 
 \operatorname{ad}(\eta)^2 x= - \langle \alpha, \eta \rangle^2 x
\ \ \text{for all } \eta \in \mathfrak{t}
\},
\\
\mathfrak{m}_0
&=
\{y \in \mathfrak{m} \mid \operatorname{ad}(\eta) y= 0
\ \ \text{for all } \eta \in \mathfrak{t}
\},
\\
\mathfrak{m}_\alpha
&=
\{y \in \mathfrak{m} \mid 
 \operatorname{ad}(\eta)^2y = - \langle \alpha, \eta \rangle^2 y
\ \ \text{for all } \eta \in \mathfrak{t}
\}.
\end{align*}
These are just the eigenspace decompositions of the commuting operators $\{\operatorname{ad}(\xi)^2 \}_{\xi \in \mathfrak{t}}$. On the other hand, the following lemma concerns the eigenspace decomposition of the commuting operators $\{A^N_{dL_a(\xi)}\}_{\xi \in \mathfrak{t}}$.
\begin{lem}\label{commute3}
There exists a unique finite subset $\Lambda$ of $\mathfrak{t}$ such that 
\begin{equation*}
T_{aK} N = \sum_{\lambda \in \Lambda} S_\lambda,
\end{equation*}
where $S_\lambda$ is a nonzero subspace of $T_{aK} N$ defined by
\begin{equation*}
S_\lambda
=
\{x \in T_{aK} N \mid A^N_{dL_a(\eta)}(x) = \langle \lambda, \eta \rangle x \ \ \textup{for all} \ \eta \in \mathfrak{t} \}.
\end{equation*}
\end{lem}
\begin{proof}
By left translation we can assume $a K = eK$ without loss of generality. It is easy to see that such a subset $\Lambda$ is unique. To see the existence we take a basis $\{\eta_i\}_{i = 1}^c$ of $\mathfrak{t}$. We denote by $\{\lambda(\eta_i)_1, \cdots , \lambda(\eta_i)_{m(i)}\}$ the set of all distinct eigenvalues of the shape operator $A^N_{\eta_i}$ and by $W_{\lambda(\eta_i)_1}, \cdots , W_{\lambda(\eta_i)_{m(i)}}$ their eigenspaces. 
Since $\{A^N_{\eta_i}\}_{i = 1}^c$ is a commuting family we have the decomposition
\begin{equation*}
T_{eK} N = \sum_{j_1 = 1}^{m(1)} \cdots \sum_{j_c = 1}^{m(c)} \ 
(W_{\lambda(\eta_1)_{j_1}} \cap \cdots \cap W_{\lambda(\eta_c)_{j_c}}).
\end{equation*}
Define a linear functional $\lambda_{j_1 \cdots j_c} : \mathfrak{t} \rightarrow \mathbb{R}$ by
\begin{equation*}
\lambda_{j_1 \cdots j_c}(a_1 \eta_1 + \cdots + a_c \eta_c) 
= 
a_1\lambda(\eta_1)_{j_1} + \cdots + a_c \lambda(\eta_c)_{j_c}
, \quad
\text{where} \ a_1, ... , a_c \in \mathbb{R}.
\end{equation*}
Then for each $\eta = a_1 \eta_1 + \cdots + a_c \eta_c \in \mathfrak{t}$ and $x \in W_{\lambda(\eta_1)_{j_1}} \cap \cdots \cap W_{\lambda(\eta_c)_{j_c}}$ we have 
\begin{align*}
A^N_\eta(x)
= 
(a_1 A^N_{\eta_1} + \cdots + a_c A^N_{\eta_c})(x)
= 
\lambda_{j_1 \cdots j_c}(\eta) x.
\end{align*}
Set 
$\Lambda := \{\lambda_{j_1 \cdots j_c}\}_{1 \leq j_1 \leq m(1) ,\  \cdots, \  1 \leq  j_c \leq m(c)}$ 
and
$S_{\lambda_{j_1 \cdots j_c}} := W_{\lambda(\eta_1)_{j_1}} \cap \cdots \cap W_{\lambda(\eta_c)_{j_c}}$.
Identifying $\mathfrak{t}$ with the dual space $\mathfrak{t}^*$ we obtain the desired subset $\Lambda \subset \mathfrak{t}$ and the decomposition 
$T_{eK} N = \sum_{\lambda \in \Lambda} S_\lambda$. This proves the lemma.
\end{proof}

The following proposition concerns the eigenspace decomposition of the union of commuting operators $\{R_{dL_a(\xi)}\}_{\xi \in \mathfrak{t}} \cup \{A^N_{dL_a(\xi)}\}_{\xi \in \mathfrak{t}}$. 
\begin{prop}\label{cca}
Let $\Lambda$ be as in Lemma \textup{\ref{commute3}}. Then the tangent space and the normal space of $N$ are decomposed as follows:
\begin{align}
\label{cc1}
T_{aK} N 
&=
\sum_{\lambda \in \Lambda_0} (dL_a (\mathfrak{m}_{0}) \cap S_\lambda) 
+ 
\sum_{\alpha \in \Delta^+} \sum_{\lambda \in \Lambda_\alpha} (dL_a(\mathfrak{m}_{\alpha}) \cap S_\lambda),
\\
\label{cc2}
T_{aK}^\perp N
&=
dL_a(\mathfrak{m}_{0}) \cap T^\perp_{aK} N + \sum_{\alpha \in \Delta^+} (dL_a(\mathfrak{m}_{\alpha}) \cap T^\perp_{aK} N),
\end{align}
where $\Lambda_0 := \{\lambda \in \Lambda \mid dL_a(\mathfrak{m}_0) \cap S_\lambda \neq \{0\}\}$ and
$\Lambda_\alpha :=\{\lambda \in \Lambda \mid dL_a(\mathfrak{m}_\alpha) \cap S_\lambda \neq \{0\}\}$.
\end{prop}
\begin{proof}
By left translation we can assume $aK = eK$ without loss of generality. Since the tangent space is invariant under $\{R_\xi\}_{\xi \in \mathfrak{t}}$ the normal space is also invariant under $\{R_\xi\}_{\xi \in \mathfrak{t}}$ and we have the decompositions
\begin{align*}
T_{eK} N 
&=
\mathfrak{m}_{0} \cap T_{eK} N
+ 
\sum_{\alpha \in \Delta^+} 
(\mathfrak{m}_{\alpha} \cap  T_{eK} N ),
\\
T_{eK}^\perp N
&=
\mathfrak{m}_{0} \cap T^\perp_{eK} N + \sum_{\alpha \in \Delta^+} (\mathfrak{m}_{\alpha} \cap T^\perp_{eK} N).
\end{align*}
By the curvature-adapted property, $\mathfrak{m}_0 \cap T_{eK} N$ and $\mathfrak{m}_\alpha \cap T_{eK} N$ are invariant under $\{A^N_\xi\}_{\xi \in \mathfrak{t}}$. Thus by similar arguments as in the proof of Lemma \ref{commute3} we obtain the eigenspace decompositions
\begin{equation*}
\mathfrak{m}_0 \cap T_{eK} N
=
\sum_{\lambda \in \Lambda_0}
(\mathfrak{m}_0 \cap S_\lambda)
, \quad
\mathfrak{m}_\alpha \cap T_{eK} N
=\sum_{\lambda \in \Lambda_\alpha}
(\mathfrak{m}_\alpha \cap S_\lambda)
\end{equation*}
and the assertion follows.
\end{proof}

\begin{example}\label{formulation}
Let $H \curvearrowright M$ be a Hermann action of cohomogeneity $c$. Choose a maximal abelian subspace $\mathfrak{t}$ in $\mathfrak{m} \cap \mathfrak{p}$. Take $w \in \mathfrak{t}$, set $a := \exp w$ and consider the orbit $N = H \cdot a K$ through $aK$. From the decomposition \eqref{tangent2} it is clear that $\mathfrak{t}$ is a $c$-dimensional abelian subspace in $\mathfrak{m}$ satisfying the condition (ii) of Definition \ref{ccad} for any $v \in \{dL_a(\xi)\}_{\xi \in \mathfrak{t}}$. We set 
$U(1)_0 := \{\epsilon \in U(1)_{\geq 0} \mid \mathfrak{m}_{0, \epsilon} \neq \{0\}\}$ and
\begin{equation*}
U(1)_0^\top := \{\epsilon \in U(1)_{0} \mid   \epsilon \neq 1\}.
\end{equation*}
We also set $U(1)_\alpha := \{\epsilon \in U(1) \mid \mathfrak{m}_{\alpha, \epsilon} \neq \{0\}\}$ and
\begin{align*}
&
U(1)^\top_{\alpha} 
:= \{\epsilon \in U(1)_\alpha \mid \langle \alpha , w \rangle + \frac{1}{2} \arg \epsilon \notin \pi \mathbb{Z}
\},
\\
&
U(1)^\perp_{\alpha} 
:= \{\epsilon \in U(1)_\alpha \mid  \langle \alpha , w \rangle + \frac{1}{2} \arg \epsilon \in \pi \mathbb{Z}
\}.
\end{align*}
Then we can rewrite the decompositions \eqref{tangent2} and \eqref{normal2} as follows:
\begin{align}
\label{tangent3}
T_{aK}N
&=
\sum_{\epsilon  \in U(1)^\top_{0}} dL_{a}(\mathfrak{m}_{0, \epsilon} )
+
\sum_{\alpha \in \Delta^+}
\sum_{\epsilon \in U(1)_\alpha^\top }
dL_{a}(\mathfrak{m}_{\alpha, \epsilon})
,\\
\label{normal3}
T^\perp_{aK}N
&=
\ \qquad dL_{a}(\mathfrak{t})  \qquad 
+ 
\sum_{\alpha \in \Delta^+}
\sum_{\epsilon \in U(1)^\perp_\alpha} dL_{a}(\mathfrak{m}_{\alpha, \epsilon}).
\end{align}
For each $\alpha \in \Delta^+$ and $\epsilon \in U(1)_\alpha^\top$ we set
\begin{equation*}
\lambda(\alpha, \epsilon)
:=
- \cot \left(\langle \alpha, w \rangle + \frac{1}{2} \arg \epsilon \right)  \alpha 
\ \in \mathfrak{t}.
\end{equation*}
Then $\Lambda_0$ and $\Lambda_\alpha$ in Proposition \ref{cca} are 
\begin{align*}
\Lambda_0 
\left\{
\begin{array}{lcl}
= &\{0\} & (\text{if \ \ $U(1)^\top_0 \neq \emptyset$})
\\
= & \emptyset & (\text{if \ \ $U(1)^\top_0 = \emptyset$})
\end{array}
\right.
, \qquad 
\Lambda_\alpha 
=
\{\lambda(\alpha, \epsilon) \mid \epsilon \in U(1)_\alpha^\top\}.
\end{align*}
Note that the correspondence $U(1)_\alpha^\top \ni \epsilon \mapsto \lambda(\alpha, \epsilon) \in \Lambda_\alpha$ is one-to-one because 
$\cot x$ is strictly decreasing on $\mathbb{R}/\pi \mathbb{Z}$. Thus we have 
\begin{equation*}
\begin{array}{ccc}
{\displaystyle
\sum_{\substack{\epsilon \in U(1)^\top_{0}}}
dL_a(\mathfrak{m}_{0, \epsilon}) = dL_a(\mathfrak{m}_0) \cap S_0,}
&&
dL_a(\mathfrak{m}_{\alpha, \epsilon})
=
dL_a(\mathfrak{m}_{\alpha, \epsilon}) \cap S_{\lambda(\alpha, \epsilon)},
\medskip
\\
dL_a(\mathfrak{t}) = dL_a(\mathfrak{m}_0) \cap T^\perp_{aK} N,
&&
{\displaystyle
\sum_{\epsilon \in U(1)^\perp_\alpha}dL_a(\mathfrak{m}_{\alpha, \epsilon})
=
dL_a(\mathfrak{m}_\alpha ) \cap T^\perp_{aK} N }
\end{array}
\end{equation*}
and therefore the decompositions \eqref{tangent3} and \eqref{normal3} are expressed as
\begin{align}
\label{tangent4}
T_{aK} N 
&=
dL_a(\mathfrak{m}_0) \cap S_0
+
\sum_{\alpha \in \Delta^+} \sum_{\lambda \in \Lambda_\alpha}
(dL_a(\mathfrak{m}_\alpha) \cap S_\lambda),
\\
\label{normal4}
T^\perp_{aK}N
&=
dL_a(\mathfrak{m}_0) \cap T^\perp_{aK} N
+
\sum_{\alpha \in \Delta^+}
(dL_a(\mathfrak{m}_\alpha) \cap T^\perp_{aK}N).
\end{align}
\end{example}

%%%%%%%
\section{Principal curvatures via the parallel transport map}\label{pcvptm}
Let $M = G/K$ be a symmetric space of compact type and $\Phi_K: V_\mathfrak{g} \rightarrow M$ the parallel transport map. In \cite{Koi02} and \cite{M2} an explicit formula for the principal curvatures of the PF submanifold $\Phi_K^{-1}(N)$ of $V_\mathfrak{g}$ was given under the assumption that $N$ is a curvature-adapted submanifold of $M$. In this section we refine that formula to the case of $c$-curvature-adapted submanifolds so that orbits of Hermann actions can be applied. 

Let $N$ be a $c$-curvature-adapted submanifold of $M$. To consider the PF submanifold $\Phi_K^{-1}(N)$ of $V_\mathfrak{g}$ we can assume $eK \in N$ without loss of generality due to the equivariant property \eqref{equiv} of $\Phi_K$. Choose and fix an arbitrary $c$-dimensional abelian subspace  $\mathfrak{t}$ in $T^\perp_{eK} N$ satisfying the condition (ii) of Definition \ref{ccad} for some $v \in T^\perp_{eK} N$. Recall the decompositions given in Proposition \ref{cca}:
\begin{align}
\label{decomp11}
T_{eK} N 
&=
\sum_{\lambda \in \Lambda_0} (\mathfrak{m}_{0} \cap S_\lambda) 
+ 
\sum_{\alpha \in \Delta^+} \sum_{\lambda \in \Lambda_\alpha} (\mathfrak{m}_{\alpha} \cap S_\lambda),
\\
\label{decomp12}
T_{eK}^\perp N
&=
\mathfrak{m}_{0} \cap T^\perp_{eK} N + \sum_{\alpha \in \Delta^+} (\mathfrak{m}_{\alpha} \cap T^\perp_{eK} N).
\end{align}
Set
\begin{equation*}
\begin{array}{lll}
m(0,\lambda) :=\dim (\mathfrak{m}_0 \cap S_\lambda)
, &&
m(\alpha,\lambda) :=\dim (\mathfrak{m}_\alpha \cap S_\lambda),
\medskip
\\
m(0,\perp) :=\dim (\mathfrak{m}_0 \cap T^\perp_{eK} N)
, &&
m(\alpha,\perp) :=\dim (\mathfrak{m}_{\alpha} \cap T^\perp_{eK} N).
\end{array}
\end{equation*}
Take bases
\begin{equation*}
\begin{array}{lll}
\{y^{0, \lambda}_j\}_{j = 1}^{m(0, \lambda)} \ \text{of}\  \mathfrak{m}_0 \cap S_\lambda
, &&
 \{y^{\alpha, \lambda}_k\}_{k = 1}^{m(\alpha, \lambda)} \ \text{of}\  \mathfrak{m}_\alpha \cap S_\lambda,
\medskip
\\
\{y^{0, \perp}_l\}_{l = 1}^{m(0, \perp)} \ \text{of} \ \mathfrak{m}_0 \cap T^\perp_{eK} N
, &&
\{y^{\alpha, \perp}_r\}_{r = 1}^{m(\alpha, \perp)}\  \text{of} \  \mathfrak{m}_\alpha \cap T^\perp_{eK} N.
\end{array}
\end{equation*}
Then we obtain a basis
\begin{equation*}
\bigcup_{\lambda \in \Lambda_0}
\{y^{0, \lambda}_j\}_{j = 1}^{m(0, \lambda)}
\ \cup \ 
\{y^{0, \perp}_l\}_{l = 1}^{m(0, \perp)}
\end{equation*}
of $\mathfrak{m}_0$ and a basis
\begin{equation*}
\bigcup_{\lambda \in \Lambda_\alpha}
\{y^{\alpha, \lambda}_k\}_{k = 1}^{m(\alpha, \lambda)}
\ \cup \ 
\{y^{\alpha, \perp}_r\}_{r = 1}^{m(\alpha, \perp)}
\end{equation*}
of $\mathfrak{m}_\alpha$. Via an isometry $\psi_\alpha: \mathfrak{m}_\alpha \rightarrow \mathfrak{k}_\alpha$ defined by \eqref{psi} we take a basis
\begin{equation*}
\bigcup_{\lambda \in \Lambda_\alpha}
\{x^{\alpha, \lambda}_k\}_{k = 1}^{m(\alpha, \lambda)}
\ \cup \ 
\{x^{\alpha, \perp}_r\}_{r = 1}^{m(\alpha, \perp)}
\end{equation*}
of $\mathfrak{k}_\alpha$. Finally we choose a basis $\{x_i^0\}_{i= 1}^{\dim \mathfrak{k}_0}$ of $\mathfrak{k}_0$. 
Then the relations
\begin{equation*}
\begin{array}{lcl}
\ [\xi, x^0_i] =0, 
&& 
\ [\xi, y^{0, \lambda}_j] = [\xi, y^{0, \perp}_l]= 0,
\medskip
\\
\ [\xi, x^{\alpha, \lambda}_k] = - \langle \alpha, \xi \rangle\,  y^{\alpha, \lambda}_k,
&&
\ [\xi, y^{\alpha, \lambda}_k] =  \langle \alpha, \xi \rangle\,  x^{\alpha, \lambda}_k, 
\medskip
\\
\ [\xi, x^{\alpha, \perp}_r] = - \langle \alpha, \xi \rangle\, y^{\alpha, \perp}_r,
&& 
\ [\xi, y^{\alpha, \perp}_r] = \langle \alpha, \xi \rangle\,  x^{\alpha, \perp}_r.
\end{array}
\end{equation*}
hold for any $\xi \in \mathfrak{t}$.

We write $V(\mathfrak{g})$ for the Hilbert space $V_\mathfrak{g} = L^2([0,1], \mathfrak{g})$ and decompose
\begin{align*}
V(\mathfrak{g})
&=
V(\mathfrak{k}_0)
+
V(\mathfrak{m}_0 \cap T_{eK} N)
+
V(\mathfrak{m}_0 \cap T^\perp_{eK} N)
\\
&
+
\sum_{\alpha \in \Delta^+} ( V(\mathfrak{k}_\alpha) + V(\mathfrak{m}_\alpha \cap T_{eK}N) + V(\mathfrak{m}_\alpha \cap T^\perp_{eK} N)).
\end{align*}
We equip a suitable basis with each term above. Recall that in addition to
\begin{equation*}
\{1, \sqrt{2} \cos 2 n \pi t, \sqrt{2} \cos 2 n \pi t\}_{n = 1}^\infty
\end{equation*}
there are two other kinds of orthonormal bases of $L^2([0,1], \mathbb{R})$, namely
\begin{align*}
\{1, \sqrt{2} \cos 2 n \pi t\}_{n = 1}^\infty
\qquad \text{and} \qquad 
\{\sqrt{2} \sin n \pi t\}_{n = 1}^\infty.
\end{align*}
We consider bases 
\begin{equation*}
\begin{array}{ccl}
\{x_i^0 \sin n \pi t\}_{i, \, n} &\text{of}& V(\mathfrak{k}_0),
\medskip
\\
\{y_j^{0, \lambda}\}_{\lambda, \, j}\cup\{y_j^{0, \lambda} \cos n \pi t\}_{\lambda, \, j, \, n} &\text{of}&V(\mathfrak{m}_0 \cap T_{eK} N),
\medskip
\\
\{y^{0, \perp}_l\}_l \cup \{y^{0, \perp}_l \cos n \pi t\}_{l, \, n} &\text{of}& V(\mathfrak{m}_0
 \cap T^\perp_{eK} N),
\medskip
\\
\{x_k^{\alpha, \lambda} \sin n \pi t\}_{\lambda, \, k, \,n} \cup \{x^{\alpha, \perp}_{r} \sin n \pi t\}_{r, \, n}&\text{of}&V(\mathfrak{k}_\alpha),
\medskip
\\
\{y^{\alpha, \lambda}_k\}_{\lambda, \, k} \cup \{y^{\alpha, \lambda}_k \cos n \pi t\}_{\lambda, \, n,\,k}&\text{of}&V(\mathfrak{m}_\alpha \cap T_{eK} N),
\medskip
\\
\{y^{\alpha, \perp}_r\}_r\cup \{y^{\alpha, \perp}_r \cos n \pi t\}_{n,\,r} &\text{of}&V(\mathfrak{m}_\alpha \cap T^\perp_{eK} N).
\end{array}
\end{equation*}
Then all these bases form a basis of $V(\mathfrak{g}) = V_\mathfrak{g} \cong T_{\hat{0}} V_\mathfrak{g}$. Since $\Phi : V_\mathfrak{g} \rightarrow G$ is a Riemannian submersion with the orthogonal direct sum decomposition (\cite[p.\ 686]{TT95})
\begin{equation*}
T_{\hat{0}} V_\mathfrak{g}  = T_{\hat{0}} \Phi^{-1}(e) \oplus \hat{\mathfrak{g}}
, \qquad
{\textstyle X = (X - \int_0^1 X(t) dt) \oplus \int_0^1 X(t) dt},
\end{equation*}
we have the orthogonal direct sum decomposition
\begin{equation*}
T_{\hat{0}} V_\mathfrak{g}
\cong
T_{\hat{0}} \Phi_K^{-1}(N)
\oplus
T^\perp_{eK} N 
, \qquad
X 
=
{\textstyle (X - (\int_0^1 X(t) dt)^\perp)
\oplus
(\int_0^1 X(t) dt)^\perp},
\end{equation*}
where $\perp$ denotes the projection from $\mathfrak{g} = \mathfrak{k} \oplus T_{eK} N \oplus T^\perp_{eK} N$ onto $T^\perp_{eK} N$. Thus we obtain a basis
\begin{align*}
&
\{x_i^0 \sin n \pi t\}_{i, \, n}
\cup
\{y_j^{0, \lambda}\}_{\lambda, \,j} 
\cup
\{y^{0, \lambda}_j \cos n \pi t\}_{\lambda, \,j, \, n}
\cup
\{y^{0, \perp}_r \cos n \pi t\}_{r, \, n}
\\
& 
\ \cup \  
\bigcup_{\alpha \in \Delta^+}
(
\{x_k^{\alpha, \lambda} \sin n \pi t\}_{\lambda,\,  k, \,n}
\cup
\{y^{\alpha, \lambda}_k\}_{\lambda, \, k} 
\cup
\{y^{\alpha, \lambda}_k \cos n \pi t\}_{\lambda, \, k, \, n}
)
\\
&
\ \cup \ 
\bigcup_{\alpha \in \Delta^+}
(
\{x^{\alpha, \perp}_r \sin n \pi t\}_{r,n}
\cup
\{y^{\alpha, \perp}_r \cos n \pi t\}_{r, \, n} 
)
\end{align*}
of $T_{\hat{0}} \Phi_K^{-1}(N)$.

For each $\xi \in \mathfrak{t}$ we denote by $A^{\Phi_K^{-1}(N)}_{\hat{\xi}}$ the shape operator of $\Phi_K^{-1}(N)$ in the direction of $\hat{\xi}$. Similarly to \cite[Lemma 3.1]{M2} the following lemma holds.
\begin{lem}\label{lem} \ 
\begin{enumerate}
\item 
$
\displaystyle
A^{\Phi_K^{-1}(N)}_{\hat{\xi}}(x^0_i \sin n \pi t) = 0
$
,\ 
$
\displaystyle
A^{\Phi_K^{-1}(N)}_{\hat{\xi}}(y^{0, \lambda}_j)
=
\langle \lambda, \xi \rangle y^{0, \lambda}_j,
$

\item 
$
\displaystyle
A^{\Phi_K^{-1}(N)}_{\hat{\xi}}(y^{0, \lambda}_j \cos n \pi t)
= 
A^{\Phi_K^{-1}(N)}_{\hat{\xi}}(y^{0, \perp}_l \cos n \pi t)
=0,
$

\item
$
\displaystyle
A^{\Phi_K^{-1}(N)}_{\hat{\xi}}
(x^{\alpha, \perp}_r \sin n \pi t)
=
- \frac{\langle \alpha, \xi \rangle}{n \pi} y^{\alpha, \perp}_r  \cos n \pi t,
$
\\
$
\displaystyle
A^{\Phi_K^{-1}(N)}_{\hat{\xi}}
(y^{\alpha, \perp}_r \cos n \pi t)
= 
- \frac{\langle \alpha, \xi \rangle}{n \pi} x^{\alpha, \perp}_r \sin n \pi t,
$

\item
$
\displaystyle
A^{\Phi_K^{-1}(N)}_{\hat{\xi}}(y^{\alpha, \lambda}_k)
=
\langle \lambda , \xi \rangle y^{\alpha, \lambda}_k
+
\frac{2 \langle \alpha, \xi \rangle}{\pi} \sum_{n = 1}^\infty \frac{1}{n} (x^{\alpha, \lambda}_k\sin n \pi t ) ,
$

\item
$
\displaystyle
A^{\Phi_K^{-1}(N)}_{\hat{\xi}}(x^{\alpha, \lambda}_k \sin n \pi t)
=
- \frac{\langle \alpha, \xi \rangle}{n \pi } y^{\alpha, \lambda}_k (-1 + \cos n \pi t) ,
$

\item
$
\displaystyle
A^{\Phi_K^{-1}(N)}_{\hat{\xi}}(y^{\alpha, \lambda}_k \cos n \pi t)
=
- \frac{\langle \alpha, \xi \rangle}{n \pi} x^{\alpha, \lambda}_k \sin n \pi t
$.
\end{enumerate}
\end{lem}

The following theorem describes the principal curvatures of the PF submanifold $\Phi_K^{-1}(N)$ of $V_\mathfrak{g}$. This theorem refines \cite[Theorem 3.2]{M2} (see also \cite[Theorem 3.3]{Koi02}). In fact, if $c = 1$ then it is equivalent to the original one. It can be proven by the similar arguments using Lemma \ref{lem}.
\begin{thm}\label{pc1} 
Let 
$M = G/K$ be a symmetric space of compact type, 
$\Phi_K: V_\mathfrak{g} \rightarrow M$ the parallel transport map, 
$N$ a $c$-curvature-adapted submanifold of $M$ through $e K$,
and 
$\mathfrak{t}$ an arbitrary $c$-dimensional abelian subspace in $\mathfrak{m}$ satisfying the condition \textup{(ii)} of Definition \textup{\ref{ccad}}. Then for each $\xi \in \mathfrak{t}$ the principal curvatures of the PF submanifold $\Phi^{-1}_K(N)$ in the direction of $\hat{\xi}$ are given by 
\begin{align*}
\{0\}
& \cup
\{\langle \lambda ,\xi \rangle 
\mid
\lambda \in \Lambda_0 \cup {\textstyle \bigcup_{\beta \in \Delta^+_\xi }} \Lambda_\beta \}
\\
& \cup
\left\{
\left.
\frac{\langle \alpha, \xi \rangle}{\arctan \frac{\langle \alpha, \xi \rangle}{\langle \lambda, \xi \rangle}+ m \pi}
\ \right|\ 
\alpha \in \Delta^+ \backslash \Delta^+_\xi , \ \lambda \in \Lambda_\alpha, \ m \in \mathbb{Z} \right\}
\\
& \cup
\left\{
\left.
\frac{\langle \alpha, \xi \rangle}{n \pi}
\ \right|\ 
\alpha \in \Delta^+ \backslash \Delta^+_\xi , \ \mathfrak{m}_\alpha \cap T^\perp_{eK} N \neq \{0\}, \  n \in \mathbb{Z} \backslash \{0\}
\right\},
\end{align*}
where we set $\Delta^+_\xi := \{\beta \in \Delta^+ \mid \langle \beta , \xi \rangle = 0\}$ and $\arctan \frac{\langle \alpha, \xi \rangle}{\langle \lambda, \xi \rangle} := \frac{\pi}{2}$ if $\langle \lambda, \xi \rangle = 0$. The eigenfunctions and the multiplicities are given in the following table.
\begin{table}[h]
\begin{tabular}{lll} \hline
	Eigenvalue 
	&
	Basis of eigenfunctions 
	& 
	Multiplicity
\\ \hline 
	$0$ 
	&
	\hspace{-3.5mm} 
	$\begin{array}{l}
		\{
		x^0_i \sin n \pi t,\ y^{0, \lambda}_j \cos n \pi t, \  y^{0, \perp}_l \cos n \pi t
		\}_{\lambda \in \Lambda_0, \, n \in \mathbb{Z}_{\geq 1} , \, i, \, j, \, l} 
	\\ 
	\ \cup \ 
		\{
		x^{\beta, \lambda}_k \sin n \pi t
		, \ 
		y^{\beta, \lambda}_k \cos n \pi t
		\}_{\beta \in \Delta_\xi, \, \lambda \in \Lambda_\beta, \, n \in \mathbb{Z}_{\geq 1}, \, k}
	\\
	\ \cup\ 
		\{
		x^{\beta, \perp}_r \sin n \pi t
		, \ 
		y^{\beta, \perp}_r \cos n \pi t
		\}_{\beta \in \Delta_\xi , \ n \in \mathbb{Z}_{\geq 1}, \, r}
	\end{array}$
	&
	$\infty$  
\\
	$\langle \lambda, \xi \rangle$ 
	& 	
	$\{y^{0, \lambda}_j\}_j \cup \{y^{\beta, \lambda}_k\}_{\beta \in \Delta_\xi,\,  k}$
	&
	\hspace{-3.5mm} 
	$\begin{array}{l}
		m(0, \lambda)\ +\  \smallskip
		\\
		\sum_{\beta} m(\beta, \lambda)
	\end{array}$
	%$m(0, \lambda)+\sum_{\beta} m(\beta, \lambda)$
\\ 
	$\frac{\langle \alpha, \xi \rangle}{\arctan \frac{\langle \alpha, \xi \rangle}{\langle \lambda, \xi \rangle}+ m \pi}$
	&
	$
	\{\underset{n \in \mathbb{Z}}{\sum}
	\frac{ \arctan \frac{\langle \alpha, \xi \rangle}{\langle \lambda, \xi \rangle} + m \pi}{\arctan \frac{\langle \alpha, \xi\rangle}{\langle \lambda, \xi \rangle} + (m + n)\pi } 
	(x^{\alpha, \lambda}_k \sin n \pi t
	+
	y^{\alpha, \lambda}_k \cos n \pi t)
	\}_{k}
	$ 
	&
	$m(\alpha, \lambda)$
\\
	$\frac{\langle \alpha, \xi \rangle}{n \pi}$
	& 
	$\{x^{\alpha, \perp}_r \sin n \pi t - y^{\alpha, \perp}_r \cos n \pi t
	\}_r$ 
	&
	$m(\alpha,\perp)$
\\ \hline 
\end{tabular}
\end{table}
\end{thm}

%%%%%%%
\section{Principal curvatures of $P(G, H \times K)$-orbits}\label{pcpo}

In this section, from Theorem \ref{pc1} we derive an explicit formula for the principal curvatures of orbits of $P(G, H \times K)$-actions induced by Hermann actions.

Let $M= G/K$ be a symmetric space of compact type and $H$ a symmetric subgroup of $G$. Choose and fix a maximal abelian subspace $\mathfrak{t}$ in $\mathfrak{m} \cap \mathfrak{p}$. Then $\pi(\exp \mathfrak{t})$ is a section of the Hermann action $H \curvearrowright M$ and $\hat{\mathfrak{t}} = \{\hat{x} \mid x \in \mathfrak{t}\}$ is a section of the hyperpolar $P(G, H \times K)$-action on $V_\mathfrak{g}$. We take arbitrary $w , \xi \in \mathfrak{t}$ and consider the principal curvatures of $P(G, H \times K) * \hat{w}$ in the direction of $\hat{\xi}$. 

Recall that the tangent space and the normal space of the orbit $N = H \cdot aK$ where $a := \exp w$ are decomposed as follows (cf.\  Section \ref{Hermann} and Example \ref{formulation}):
\begin{align*}
T_{aK}N
&=
\sum_{\epsilon  \in U(1)^\top_{0}} dL_{a}(\mathfrak{m}_{0, \epsilon} )
+
\sum_{\alpha \in \Delta^+}
\sum_{\epsilon \in U(1)_\alpha^\top }
dL_{a}(\mathfrak{m}_{\alpha, \epsilon})
,\\
T^\perp_{aK}N
&=
\ \qquad dL_{a}(\mathfrak{t})  \qquad 
+ 
\sum_{\alpha \in \Delta^+}
\sum_{\epsilon \in U(1)^\perp_\alpha} dL_{a}(\mathfrak{m}_{\alpha, \epsilon}),
\end{align*}
where we set $U(1)_\alpha := \{\epsilon \in U(1) \mid \mathfrak{m}_{\alpha, \epsilon} \neq \{0\}\}$ and 
\begin{align*}
&
U(1)^\top_{\alpha} 
:= \{\epsilon \in U(1)_\alpha \mid \langle \alpha , w \rangle + \frac{1}{2} \arg \epsilon \notin \pi \mathbb{Z}
\},
\\
&
U(1)^\perp_{\alpha} 
:= \{\epsilon \in U(1)_\alpha \mid  \langle \alpha , w \rangle + \frac{1}{2} \arg \epsilon \in \pi \mathbb{Z}
\}.
\end{align*}
Here $dL_a (\mathfrak{m}_{0, \epsilon})$ and $dL_a(\mathfrak{m}_{\alpha, \epsilon})$ are the eigenspaces of the shape operator $A^N_{dL_a(\xi)}$ associated with the eigenvalues $0$ and $- \langle \alpha, \xi\rangle \cot ( \langle\alpha, w \rangle + \frac{1}{2} \arg \epsilon)$ respectively. 

Using the above information we can describe the principal curvatures of orbits of $P(G, H \times K)$-actions induced by Hermann actions:
\begin{thm}\label{thm1}
Let $M = G/K$ be a symmetric space of compact type and $H$ a symmetric subgroup of $G$. Take a maximal abelian subspace $\mathfrak{t}$ in $\mathfrak{m} \cap \mathfrak{p}$ and $w \in \mathfrak{t}$. Then for each $\xi \in \mathfrak{t}$ the principal curvatures of $P(G,H \times K) * \hat{w}$ in the direction of $\hat{\xi}$ are given by
\begin{align*}
\{0\} 
& \cup
\left\{
\left.
\frac{\langle \alpha, \xi \rangle}{- \langle \alpha, w \rangle - \frac{1}{2} \arg \epsilon + m \pi}
 \ \right| \ 
\alpha \in \Delta^+ \backslash \Delta^+_\xi
, \ 
\epsilon \in U(1)_\alpha^\top
, \ 
m \in \mathbb{Z}
\right\}
\\
& \cup 
\left\{
\left.\frac{\langle \alpha, \xi \rangle}{n \pi} 
\ \right| \ 
\alpha \in \Delta^+\backslash \Delta^+_\xi
\textup{ satisfying }
U(1)_\alpha^\perp \neq \emptyset
, \ 
n \in \mathbb{Z} \backslash \{0\}
\right\}.
\end{align*}
Taking bases $\{x^{0}_i\}_i$ of $\mathfrak{k}_{0}$, $\{x^{\alpha, \epsilon}_k\}_k$ of $\mathfrak{k}_{\alpha, \epsilon}$, $\{y^{0, \epsilon}_j\}_j$ of $\mathfrak{m}_{0, \epsilon}$, $\{\eta_l\}_l$ of $\mathfrak{t}$ and $\{y^{\alpha, \epsilon}_k\}_k$ of $\mathfrak{m}_{\alpha, \epsilon}$ with the relation \eqref{basis-relation} we can describe the eigenfunctions and the multiplicities as in the following table. Here we are identifying $T_{\hat{w}} V_\mathfrak{g}$ with $T_{\hat{0}} V_\mathfrak{g}$ via the gauge transformation $g*: V_\mathfrak{g} \rightarrow V_\mathfrak{g}$ for a unique $g \in P(G,G \times \{e\})$ satisfying $g * \hat{0}  = \hat{w}$.

\begin{table}[h]
\begin{center}
\begin{tabular}{lll} \hline
	Eigenvalue 
	& 
	Basis of eigenfunctions 
	& 
	Multiplicity
\\ \hline 
	$0$ 
	& 
	\hspace{-3.5mm} 
	$\begin{array}{l}
	\{
	x^0_i \sin n \pi t, \ y^{0, \epsilon}_j \sin n \pi t, \ \eta_l \cos n \pi t
	\}_{\epsilon \in U(1)_0^\top, \, n \in \mathbb{Z}_{\geq 1}, i, \, j, \, l} 
	\\
	\ \cup \ 
	\{
	x^{\beta, \epsilon}_k \sin n \pi t, \ y^{\beta, \epsilon}_k \sin n \pi t
	\}_{\beta \in \Delta^+_\xi, \, \epsilon \in U(1)_\beta^\top, \, n \in \mathbb{Z}_{\geq 1}, \, k} 
	\\
	\ \cup \ 
	\{
	x^{\beta, \epsilon}_r \sin n \pi t, \  y^{\beta, \epsilon}_r \cos n \pi t
	\}_{\beta \in \Delta^+_\xi, \, \epsilon \in U(1)_\beta^\perp, \, n \in \mathbb{Z}_{\geq 1}, \, r}
	\end{array}$
	& 
	$\infty$  
\\ 
	$
	{\frac{\langle \alpha, \xi \rangle}{-\langle \alpha, w \rangle - \frac{1}{2} \arg \epsilon + m \pi}}
	$
	& 
	$\{
	\underset{n \in \mathbb{Z}}{\sum}
	\frac{\langle \alpha, w \rangle + \frac{1}{2} \arg \epsilon + m \pi}
			{\langle \alpha, w \rangle + \frac{1}{2} \arg \epsilon +(m + n )\pi} 
	(x^{\alpha, \epsilon}_k \sin n \pi t
	+
	y^{\alpha, \epsilon}_k \cos n \pi t)
	\}_{k}$ 
	& 
	$m(\alpha, \epsilon)$
\\ 
	${ \frac{\langle \alpha, \xi \rangle}{n \pi}}$
	&
	$
	\{x^{\alpha, \epsilon}_r \sin n \pi t - y^{\alpha, \epsilon}_r \cos n \pi t
	\}_{\epsilon \in U(1)_\alpha^\perp, \, r}
	$ 
	& 
	$\sum_\epsilon m(\alpha,\epsilon)$
\\ \hline 
\end{tabular}
\end{center}
\end{table}

In particular, if $w \in \mathfrak{t}$ is a regular point then the term $\frac{\langle \alpha, \xi \rangle}{n \pi}$ vanishes. 
\end{thm}

\begin{proof}
Take a unique $g \in P(G, G \times \{e\})$ satisfying $g * \hat{0} = \hat{w}$. By \eqref{equiv2} we have $g(0) = \exp w = a$. From \eqref{equiv} the diagram
\begin{equation*}
\begin{CD}
V_\mathfrak{g} @>g*>> V_\mathfrak{g}
\\
@V \Phi_K VV @V \Phi_K VV
\\
M @>L_a>> M
\end{CD}
\end{equation*}
commutes. Thus setting $\bar{N} := L_a^{-1}(N)$ we have $g* \Phi_K^{-1}(\bar{N}) = \Phi_K^{-1}(N) = P(G, H \times K) * \hat{w}$ by \eqref{inverse}. Moreover since $w \in \mathfrak{t}$ it follows from $g * \hat{0} = \hat{w}$ that $g(t) \in \exp \mathfrak{t}$ for all $t \in [0,1]$. Thus we have $d(g * ) \hat{\xi} = g \hat{\xi} g^{-1} = \hat{\xi}$. Hence it suffices to compute the principal curvatures of $\Phi_K^{-1}(\bar{N})$ in the direction of $\hat{\xi}$. Since $\mathfrak{t}$ is a $c$-dimensional abelian subspace in $\mathfrak{m}$ satisfying the condition (ii) of Definition \ref{ccad} we can apply $\bar{N}$ to Theorem \ref{pc1}. From \eqref{tangent3} and \eqref{normal3} the tangent space and the normal space of $\bar{N}$ are
\begin{align*}
T_{eK}\bar{N}
&=
\sum_{\epsilon  \in U(1)_0^\top} \mathfrak{m}_{0, \epsilon}
+
\sum_{\alpha \in \Delta^+}
\sum_{\epsilon \in U(1)_\alpha^\top}
\mathfrak{m}_{\alpha, \epsilon}
,\\
\label{normal3}
T^\perp_{eK}\bar{N}
&=
\qquad \mathfrak{t}  \qquad 
+ \quad 
\sum_{\alpha \in \Delta^+}
\sum_{\epsilon \in U(1)_\alpha^\perp}\mathfrak{m}_{\alpha, \epsilon}.
\end{align*}
From \eqref{tangent4} and \eqref{normal4} the above decompositions are rewritten as
\begin{align*}
T_{eK} \bar{N} 
&=
\mathfrak{m}_0 \cap \bar{S}_0
+
\sum_{\alpha \in \Delta^+} \sum_{\lambda \in \Lambda_\alpha}
(\mathfrak{m}_\alpha \cap \bar{S}_\lambda),
\\
T^\perp_{eK} \bar{N}
&=
\mathfrak{m}_0 \cap T^\perp_{eK} \bar{N}
+
\sum_{\alpha \in \Delta^+}
(\mathfrak{m}_\alpha \cap T^\perp_{eK} \bar{N}),
\end{align*}
where $\bar{S}_0 := dL_a^{-1}(S_0)$ and $\bar{S}_\lambda := dL_a^{-1}(S_\lambda)$. Since $\langle \beta , \xi \rangle= 0$ implies $\langle \lambda(\beta, \epsilon), \xi \rangle = 0$ the eigenvalue $\langle \lambda , \xi \rangle$ in the theorem is equal to $0$. Moreover taking a unique $m' \in \mathbb{Z}$ satisfying $- \pi/2 < \langle \alpha, w \rangle + \frac{1}{2} \arg \epsilon + m' \pi \leq \pi /2$ we have 
\begin{equation*}
\arctan \frac{\langle \alpha, \xi \rangle}{\langle \lambda(\alpha, \epsilon), \xi \rangle}
=
- \langle \alpha, w \rangle  - \frac{1}{2} \arg \epsilon - m' \pi.
\end{equation*}
Since $m \in \mathbb{Z}$ in the theorem is arbitrary the assertion follows.
\end{proof}

Applying \eqref{tangent3:commute} and \eqref{normal3:commute} to Theorem \ref{thm1} we obtain the following corollary.
\begin{cor}\label{cor1}
Suppose that $\sigma \circ \tau = \tau \circ \sigma$. Then the principal curvatures of the orbit $P(G, H \times K) * \hat{w}$ in the direction of $\hat{\xi}$ are given by
\begin{align*}
\{0\}
& \cup
\left\{
\left.
\frac{\langle \alpha, \xi \rangle}{- \langle \alpha, w \rangle + m \pi}
 \ \right| \ 
\alpha \in \Delta^+ \backslash \Delta_\xi^+
, \  
\mathfrak{m}_\alpha \cap \mathfrak{p} \neq \{0\}
, \ 
\langle \alpha, w\rangle  \notin \pi \mathbb{Z}, \ m \in \mathbb{Z}
\right\}
\\
& \cup
\left\{
\left.
\frac{\langle \alpha, \xi \rangle}{ - \langle \alpha, w \rangle - \frac{\pi}{2} + m  \pi}
 \ \right| \ 
\alpha \in \Delta^+ \backslash \Delta_\xi^+
, \  
\mathfrak{m}_\alpha \cap \mathfrak{h} \neq \{0\}
, \  
\langle \alpha, w\rangle + \frac{\pi}{2} \notin \pi \mathbb{Z}, \ m \in \mathbb{Z}
\right\}
\\
& \cup
\left\{
\left.
\frac{\langle \alpha, \xi \rangle}{n \pi} 
\ \right| \ 
\alpha \in \Delta^+\backslash \Delta_\xi^+
, \ 
\mathfrak{m}_\alpha \cap \mathfrak{p}  \neq \{0\}
, \ 
\langle \alpha, w\rangle \in \pi \mathbb{Z}
, \ 
n \in \mathbb{Z} \backslash \{0\}
\right.
\\
&
\left.
\qquad \qquad \text{or} \ \,
\alpha \in \Delta^+ \backslash \Delta_\xi^+
, \ 
\mathfrak{m}_\alpha \cap \mathfrak{h} \neq \{0\}
, \ 
\langle \alpha, w\rangle + \frac{\pi}{2} \in \pi \mathbb{Z}
, \ 
n \in \mathbb{Z} \backslash \{0\}
\right\}.
\end{align*}
The multiplicities are respectively given by 
\begin{equation*}
\infty
, \qquad
\dim(\mathfrak{m}_\alpha \cap \mathfrak{p})
, \qquad
\dim(\mathfrak{m}_\alpha \cap \mathfrak{h})
, \qquad
\dim(\mathfrak{m}_\alpha \cap \mathfrak{p})
+
\dim(\mathfrak{m}_\alpha \cap \mathfrak{h}).
\end{equation*}
In particular, if $w \in \mathfrak{t}$ is a regular point then the term $\frac{\langle \alpha, \xi \rangle}{n \pi}$ vanishes. 
\end{cor}

Applying \eqref{tangent4} and \eqref{normal4} to Theorem \ref{thm1} we obtain the following corollary.
\begin{cor}\label{cor2}
Suppose that $\sigma = \tau$. Then the principal curvatures of the orbit $P(G, H \times K) * \hat{w}$ in the direction of $\hat{\xi}$ are given by
\begin{align*}
\{0\} 
& \cup
\left\{
\left.
\frac{\langle \alpha, \xi \rangle}{- \langle \alpha, w \rangle + m \pi}
 \ \right| \ 
\alpha \in \Delta^+\backslash \Delta_\xi^+
, \ 
\langle \alpha, w\rangle \notin \pi \mathbb{Z}, \ m \in \mathbb{Z}
\right\}
\\
& \cup
\left\{
\left.\frac{\langle \alpha, \xi \rangle}{n \pi} \ \right| \ 
\alpha \in \Delta^+\backslash \Delta_\xi^+
, \ 
\langle \alpha, w\rangle \in \pi \mathbb{Z},  \ n \in \mathbb{Z} \backslash \{0\}
\right\}.
\end{align*}
The multiplicities are respectively given by 
\begin{equation*}
\infty
, \qquad
\dim \mathfrak{m}_\alpha 
, \qquad
\dim \mathfrak{m}_\alpha.
\end{equation*}
In particular, if $w \in \mathfrak{t}$ is a regular point then the term $\frac{\langle \alpha, \xi \rangle}{n \pi}$ vanishes. 
\end{cor}

\begin{rem}\label{rem5.4}
Terng \cite{Ter89} showed that any principal orbit of the $P(G, \Delta G)$-action, where $\Delta G$ is the diagonal of $G \times G$, is an isoparametric PF submanifold of $V_\mathfrak{g}$ and computed its principal curvatures. This result was extended by Pinkall and Thorbergsson \cite{PiTh90} to the case of $P(G, K \times K)$-action, where $K$ is a symmetric subgroup of $G$. (Note that in the equation (28) of \cite{PiTh90} the term $\alpha(Y)$ should be $- \alpha(Y)$.) More generally, Koike \cite{Koi11} computed the principal curvatures of principal orbits of the $P(G, H \times K)$-action induced by a Hermann action with the assumption that the involutions $\sigma$ and $\tau$ commute (\cite[p.\ 114]{Koi11}). Theorem \ref{thm1} above does not require such  assumptions at all. 
\end{rem}

\begin{rem}\label{rem5}
For each $\alpha \in \Delta^+$ and $\epsilon \in U(1)_\alpha^\top$ it is clear that
\begin{equation*}
\left\{
\left.
\frac{\langle \alpha, \xi \rangle}{- \langle \alpha, w \rangle - \frac{1}{2} \arg \epsilon + m \pi}
 \ \right| \ 
m \in \mathbb{Z}
\right\}
=
\left\{
\left.
- \frac{\langle \alpha, \xi \rangle}{\langle \alpha, w \rangle + \frac{1}{2} \arg \epsilon + m \pi}
 \ \right| \ 
m \in \mathbb{Z}
\right\}.
\end{equation*}
We will alternatively use the latter expression to describe the principal curvatures.
\end{rem}

%%%%%%%
\section{The austere property: reduced case}\label{austere:reduced}

In this section we study the relation between the austere properties of $H$- and $P(G, H \times K)$-orbits under the assumption that the root system $\Delta$ is reduced; the non-reduced case will be dealt with in the next section. Notice that this assumption is independent of the choice of a maximal abelian subspace $\mathfrak{t}$ in $\mathfrak{m} \cap \mathfrak{p}$. The main result of this section is the following theorem (Theorem I in Introduction):

\begin{thm}\label{main1}
Let $M = G/K$ be a symmetric space of compact type and $H$ a symmetric subgroup of $G$. Suppose that the root system $\Delta$ of a maximal abelian subspace $\mathfrak{t}$ in $\mathfrak{m} \cap \mathfrak{p}$ is reduced. Then for $w \in \mathfrak{g}$ the following conditions are equivalent:
\begin{enumerate}
\item the orbit $H \cdot (\exp w) K$ through $(\exp w) K$ is an austere submanifold of $M$,
\item the orbit $P(G, H \times K) * \hat{w}$ through $\hat{w}$ is an austere PF submanifold of $V_\mathfrak{g}$.
\end{enumerate}
\end{thm}

To prove this theorem we need the following lemma. The statement (i) was essentially shown by Ohno  \cite[Proposition 13]{Ohno21}. Note that this lemma is still valid in the non-reduced case.
\begin{lem}\label{lem2}
Let $\mathfrak{t}$ be a maximal abelian subspace in $\mathfrak{m} \cap \mathfrak{p}$ and $w \in \mathfrak{t}$. Set
\begin{equation*}
U(1)_\alpha^* := \{\epsilon \in U(1)_\alpha \ | \  \langle \alpha, w \rangle + \frac{1}{2} \arg \epsilon \notin \frac{\pi}{2} \mathbb{Z}\},
\end{equation*}
which is a subset of $U(1)_\alpha^\top$. Then
\begin{enumerate}
\item \textup{(Ohno \cite{Ohno21})} the orbit $H \cdot (\exp w)K$ through $(\exp w)K$ is an austere submanifold of $M$ if and only if the set 
\begin{equation*}
\left\{
\left.
\cot \left( \langle \alpha, w \rangle + \frac{1}{2} \arg \epsilon \right) \alpha
\ \right|\ 
\alpha \in \Delta^+, \ \epsilon \in U(1)_\alpha^*
\right\}
\end{equation*}
with multiplicities is invariant under the multiplication by $(-1)$, where the multiplicity of $\cot ( \langle \alpha, w \rangle + \frac{1}{2} \arg \epsilon ) \alpha$ is defined to be $m(\alpha, \epsilon)$,

\item the orbit $P(G, H \times K) * \hat{w}$ through $\hat{w}$ is an austere PF submanifold of $V_\mathfrak{g}$ if and only if the set
\begin{equation*}
\left\{
\left.
\frac{1}{\langle \alpha, w \rangle + \frac{1}{2} \arg \epsilon + m \pi} \alpha
\ \right|\ 
\alpha \in \Delta^+, \ \epsilon \in U(1)_\alpha^*
, \ 
m \in \mathbb{Z}
\right\}
\end{equation*}
with multiplicities is invariant under the multiplication by $(-1)$, where the multiplicity of $\frac{1}{\langle \alpha, w \rangle + \frac{1}{2} \arg \epsilon + m \pi} \alpha$ is defined to be $m(\alpha, \epsilon)$.

\end{enumerate}
\end{lem}

In connection with the proof of (ii) we reprove (i) here.
\begin{proof}
(i) Set $a := \exp w$ and $N := H \cdot aK$. From the straightforward computations (\cite[pp.\ 15-16]{Ohno21}) the normal space of $N$ is expressed as
\begin{equation}\label{eq:normalspace}
T^\perp_{aK} N
= 
dL_a(\mathfrak{m} \cap \operatorname{Ad}(a)^{-1} \mathfrak{p})
=
dL_a(\ \ \bigcup_{b \in K \cap a^{-1} Ha} \operatorname{Ad}(b)\mathfrak{t}\ \ ).
\end{equation}
Thus for each $v \in T^\perp_{aK} N$ there exist $\xi \in \mathfrak{t}$ and $b \in K \cap a^{-1} H a$ such that $v = dL_a(\operatorname{Ad}(b)\xi)$. Since $b$ belongs to $a^{-1} H a$ the isometry $L_b$ leaves the submanifold $\bar{N} := L_a^{-1}N$ invariant. Moreover since $b$ belongs to $K$ the differential $dL_b$ of the isometry $L_b$ at $eK$ is identified with $\operatorname{Ad}(b)$. Thus the shape operators satisfy $A^{\bar{N}}_{\operatorname{Ad}(b)\xi}= dL_b \circ  A^{\bar{N}}_\xi \circ dL_b^{-1}$. From this we obtain 
\begin{equation*}
A^N_{dL_a(\operatorname{Ad}(b)\xi)}
=
dL_a \circ dL_b \circ dL_a^{-1} \circ  A^{N}_{dL_a (\xi)} \circ dL_a \circ dL_b^{-1} \circ dL_a^{-1}.
\end{equation*}
This shows that the eigenvalues with multiplicities of the shape operators $A^N_{v}$ and $A^N_{dL_a(\xi)}$ coincide. Thus to consider the austere property it suffices to consider normal vectors $\{dL_a(\xi)\}_{\xi \in \mathfrak{t}}$ of $N$. Thus it follows from the eigenspace decomposition \eqref{tangent2} that the orbit $H \cdot (\exp w) K$ is an austere submanifold of $M$ if and only if the set
\begin{equation*}
\left\{
\left.
\langle \alpha, \xi \rangle \cot \left( \langle \alpha, w \rangle + \frac{1}{2} \arg \epsilon \right) 
\ \right|\ 
\alpha \in \Delta^+, \ \epsilon \in U(1)_\alpha^\top
\right\}
\end{equation*}
with multiplicities is invariant under the multiplication by $(-1)$ for each $\xi \in \mathfrak{t}$. Notice that this condition is equivalent to the condition that the set 
\begin{equation*}
\left\{
\left.
\cot \left( \langle \alpha, w \rangle + \frac{1}{2} \arg \epsilon \right) \alpha 
\ \right|\ 
\alpha \in \Delta^+, \ \epsilon \in U(1)_\alpha^\top
\right\}
\end{equation*}
with multiplicities is invariant under the multiplication by $(-1)$ (cf.\ \cite[p.\ 459]{IST09}). Hence the assertion follows from the fact $\cot (\frac{\pi}{2} + \pi\mathbb{Z}) = \{0\}$.

(ii) Choose a unique $g \in P(G, G \times \{e\})$ satisfying $g * \hat{0} = \hat{w}$. Then we have $a = g(0)$ and the commutative diagram
\begin{equation}\label{diagram6.1}
\begin{CD}
V_\mathfrak{g} @>g*>> V_\mathfrak{g}
\\
@V\Phi_KVV  @V\Phi_KVV
\\
M @>L_a>> \,M.
\end{CD}
\end{equation}
Since $\Phi_K$ is a Riemannian submersion it follows from \eqref{eq:normalspace} and \eqref{diagram6.1} that each normal vector of $\Phi_K^{-1}(N)$ is expressed as $(dg* )\operatorname{Ad}(b) \hat{\xi}$ for some $\xi \in \mathfrak{t}$ and $b \in K \cap a^{-1} H a$. Denote by $\hat{b} \in \mathcal{G}$ the constant path with value $b$. Then by \eqref{equiv} we have the commutative diagram
\begin{equation*}
\begin{CD}
V_\mathfrak{g} @>\hat{b}*>> V_\mathfrak{g}
\\
@V \Phi_K VV  @V \Phi_K VV
\\
M @>L_b>> \,M,
\end{CD}
\end{equation*}
where $\hat{b} *$ is identified with $\operatorname{Ad}(b)$ acting on $V_\mathfrak{g}$ by pointwise operation. 
Since $L_b$ leaves $\bar{N}$ invariant it follows that $\hat{b}*$ leaves $\Phi_K^{-1}(\bar{N})$ invariant. Thus we have 
\begin{equation*}
A^{\Phi_K^{-1}(\bar{N})}_{\operatorname{Ad}(b)(\hat{\xi})} = (d \hat{b}*) \circ A^{\Phi_K^{-1}(\bar{N})}_{\hat{\xi}} \circ (d\hat{b} *)^{-1}.
\end{equation*}
This together with the equality $g * \Phi_K^{-1}(\bar{N}) = \Phi_K^{-1}(N)$ implies
\begin{equation*}
A^{\Phi_K^{-1}(N)}_{(d g *)(\operatorname{Ad}(b)\hat{\xi})} 
= 
(d g *) \circ  (d \hat{b}) * \circ (d g*)^{-1} \circ A^{\Phi_K^{-1}(N)}_{(dg *)(\hat{\xi})} \circ (d g *) \circ (d \hat{b} *)^{-1} \circ (d g *)^{-1}.
\end{equation*}
Thus similarly it suffices to consider normal vectors $\{d(g*) \hat{\xi}\}_{\xi \in \mathfrak{t}}$ of $\Phi_K^{-1}(N)$. Note that  $g * \hat{0} = \hat{w}$ implies $d(g*) \hat{\xi} = \hat{\xi}$ as mentioned in the proof of Theorem \ref{thm1}. Thus from Theorem \ref{thm1} and Remark \ref{rem5} it follows that the orbit $P(G, H \times K) * \hat{w}$ is an austere PF submanifold of $V_\mathfrak{g}$ if and only if the set
\begin{equation*}
\left\{
\left.
\frac{1}{\langle \alpha, w \rangle + \frac{1}{2} \arg \epsilon + m \pi} \alpha
\ \right|\ 
\alpha \in \Delta^+, \ \epsilon \in U(1)_\alpha^\top
, \ 
m \in \mathbb{Z}
\right\}
\end{equation*}
with multiplicities is invariant under the multiplication by $(-1)$. Hence the assertion follows from the fact that the set 
$\{\frac{1}{\pi/2 + m \pi} \alpha\}_{m \in \mathbb{Z}}$
with multiplicities is invariant under the multiplication by $(-1)$ due to the equality $\frac{1}{ \pi /2 + m\pi} \alpha = (-1) \times \frac{1}{ \pi/2 +(-m-1)\pi}\alpha$.
\end{proof}

We are now in position to prove Theorem \ref{main1}.

\begin{proof}[Proof of Theorem \textup{\ref{main1}}]
Take a maximal abelian subspace $\mathfrak{t}$ in $\mathfrak{m} \cap \mathfrak{p}$. Since $\pi(\exp \mathfrak{t})$ is a section of the $H$-action we can assume $w \in \mathfrak{t}$ without loss of generality. 

``(i) $\Rightarrow$ (ii)" : Let $\alpha \in \Delta^+$ and $\epsilon \in U(1)_\alpha^*$. Since the orbit $H \cdot (\exp w) K$ is austere it follows from Lemma \ref{lem2} (i) that there exist $\alpha' \in \Delta^+$ and $\epsilon' \in U(1)_{\alpha'}^*$ such that 
\begin{equation}\label{eq6.1}
\cot \left( \langle \alpha, w \rangle + \frac{1}{2} \arg \epsilon \right)  \alpha 
= 
(-1) \times \cot \left( \langle \alpha', w \rangle + \frac{1}{2} \arg \epsilon' \right)  \alpha' .
\end{equation}
Since $\cot ( \langle \alpha, w \rangle + \frac{1}{2} \arg \epsilon ) \neq 0$ and $\cot ( \langle \alpha', w \rangle + \frac{1}{2} \arg \epsilon' ) \neq 0$ it follows from the reduced property of $\Delta$ that $\alpha' = \alpha$. Moreover since the map $\epsilon \mapsto \cot (\langle \alpha, w \rangle + \frac{1}{2} \arg \epsilon)$ is injective we have $m(\alpha, \epsilon) = m(\alpha, \epsilon')$. Then we have
\begin{equation*}
\cot \left( \langle \alpha, w \rangle + \frac{1}{2} \arg \epsilon \right) 
= 
(-1) \times \cot \left( \langle \alpha, w \rangle + \frac{1}{2} \arg \epsilon'\right).
\end{equation*}
Since $\cot x$ is strictly decreasing on $\mathbb{R}/ \pi \mathbb{Z}$ there exists a unique $n \in \mathbb{Z}$ such that 
\begin{equation*}
\langle \alpha, w \rangle + \frac{1}{2} \arg \epsilon 
= 
(-1) \times \left(\langle \alpha, w \rangle +  \frac{1}{2} \arg \epsilon'\right) + n \pi.
\end{equation*}
For each $m \in \mathbb{Z}$ we set $m' : = - n - m$. Then we obtain
\begin{equation*}
\frac{1}{\langle \alpha, w \rangle + \frac{1}{2} \arg \epsilon + m \pi} \alpha
=
(-1) \times \frac{1}{\langle \alpha, w \rangle + \frac{1}{2} \arg \epsilon' + m' \pi} \alpha.
\end{equation*}
Thus by Lemma \ref{lem2} (ii) the orbit $P(G, H \times K) * \hat{w}$ is an austere PF submanifold of $V_\mathfrak{g}$.

``(ii) $\Rightarrow$ (i)": Since the orbit $P(G, H \times K) * \hat{w}$ is austere it follows from Lemma \ref{lem2} (ii) that for each $\alpha \in \Delta^+$, $\epsilon \in U(1)_\alpha^*$ and $m \in \mathbb{Z}$ there exist $\alpha' \in \Delta^+$, $\epsilon' \in U(1)_{\alpha'}^*$ and $m' \in \mathbb{Z}$ such that 
\begin{equation*}
\frac{1}{\langle \alpha, w \rangle + \frac{1}{2} \arg \epsilon + m \pi} \alpha
=
(-1) \times \frac{1}{ \langle \alpha', w \rangle + \frac{1}{2} \arg \epsilon' + m' \pi} \alpha'.
\end{equation*}
Since $\Delta$ is reduced we have $\alpha = \alpha'$. Moreover since the map $(\epsilon, m) \mapsto \frac{1}{\langle \alpha, w \rangle + (\arg \epsilon)/2 + m \pi}$ is injective we have $m(\alpha, \epsilon) = m(\alpha, \epsilon')$. Then we have
\begin{equation*}
\frac{1}{\langle \alpha, w \rangle + \frac{1}{2} \arg \epsilon + m \pi} 
=
(-1) \times \frac{1}{\langle \alpha, w \rangle + \frac{1}{2} \arg \epsilon' + m' \pi},
\end{equation*}
that is, 
\begin{equation*}
 \langle \alpha, w \rangle + \frac{1}{2} \arg \epsilon + m \pi
=
(-1) \times \left( \langle \alpha, w \rangle + \frac{1}{2} \arg \epsilon' + m' \pi \right).
\end{equation*}
Hence we have 
\begin{equation*}
\cot \left(\langle \alpha, w \rangle + \frac{1}{2} \arg \epsilon \right) \alpha
=
(-1) \times \cot \left(\langle \alpha, w \rangle + \frac{1}{2} \arg \epsilon' \right) \alpha .
\end{equation*}
Thus by Lemma \ref{lem2} (i) the orbit $H \cdot (\exp w) K$ is an austere submanifold of $M$.
\end{proof}

\begin{rem} In the above proof we essentially showed that the following conditions are equivalent when $\Delta$ is reduced: 
\begin{enumerate}
\item  the orbit $H \cdot (\exp w)K$ through $(\exp w) K$ is an austere submanifold of $M$,

\item for each $\alpha \in \Delta^+$ the set
\begin{equation*}
\left\{
\left.
\cot \left( \langle \alpha, w \rangle + \frac{1}{2} \arg \epsilon \right) \alpha
\ \right|\ 
\epsilon \in U(1)_\alpha^*
\right\}
\end{equation*}
with multiplicities is invariant under the multiplication by $(-1)$,

\item for each $\alpha \in \Delta^+$ the set
\begin{equation*}
\left\{
\left.
\frac{1}{\langle \alpha, w \rangle + \frac{1}{2} \arg \epsilon + m \pi} \alpha
\ \right|\ 
\epsilon \in U(1)_\alpha^*
, \ 
m \in \mathbb{Z}
\right\}
\end{equation*}
with multiplicities is invariant under the multiplication by $(-1)$.
\item the orbit $P(G, H \times K) * \hat{w}$ through $\hat{w}$ is an austere PF submanifold of $V_\mathfrak{g}$,

\end{enumerate}
\end{rem}

%%%%%%%
\section{The austere property: general case}\label{austere:general}

In this section, without supposing that the root system $\Delta$ is reduced, we study the relation between the austere properties of $H$- and $P(G, H \times K)$-orbits.

As a preliminary we prove the following lemma, which generalizes Lemma 4.32 in  \cite{Ika11}. In fact if $\sigma = \tau$ then it is just the original one. 
\begin{lem}\label{lem3}
Let $M = G/K$ be a symmetric space of compact type, $H$ a symmetric subgroup of $G$ and $\mathfrak{t}$ a maximal abelian subspace in $\mathfrak{m} \cap \mathfrak{p}$. Suppose that there exists $\alpha \in \Delta$ satisfying $2 \alpha \in \Delta$. Then the multiplicities satisfy $m(\alpha) > m(2 \alpha)$.
\end{lem}
\begin{proof}
We extend the inner product of $\mathfrak{g}$ to the complex symmetric bi-linear form on $\mathfrak{g}^\mathbb{C}$ which is still denoted by $\langle \cdot , \cdot \rangle$. Choose $\epsilon \in U(1)$ satisfying $\mathfrak{g}(\alpha, \epsilon) \neq \{0\}$. Since $\overline{\mathfrak{g}(\alpha, \epsilon)} = \sigma(\mathfrak{g}(\alpha, \epsilon)) = \mathfrak{g}(- \alpha, \epsilon^{-1})$ the involution $z \mapsto \sigma(\bar{z})$ leaves $\mathfrak{g}(\alpha, \epsilon)$ invariant. Thus we have the $(\pm1)$-eigenspace decomposition
\begin{equation*}
\mathfrak{g}(\alpha, \epsilon) = \mathfrak{g}(\alpha, \epsilon)^+ \oplus \mathfrak{g}(\alpha, \epsilon)^-.
\end{equation*}
Take $z_0 \in \mathfrak{g}(\alpha, \epsilon)^+ \backslash \{0\}$. Then by definition we have 
\begin{equation}\label{eq7.0}
\sigma(z_0) = \bar{z}_0, 
\quad
\sigma(\bar{z}_0) = z_0
, \quad
\tau(z_0) = \epsilon \bar{z}_0
, \quad
\tau(\bar{z}_0) = \epsilon^{-1} z_0.
\end{equation}
Since $[\mathfrak{g}(\alpha), \mathfrak{g}(\alpha)] \subset \mathfrak{g}(2\alpha)$ we have the linear map $\operatorname{ad}(z_0): \mathfrak{g}(\alpha) \rightarrow \mathfrak{g}(2 \alpha)$. We restrict this map to the subspace
\begin{equation*}
(\mathbb{C} z_0)^\perp := \{z \in \mathfrak{g}(\alpha) \mid \langle z, \bar{z}_0\rangle = 0\}.
\end{equation*}
It suffices to show that the restriction $\operatorname{ad}(z_0): (\mathbb{C}z_0)^\perp \rightarrow \mathfrak{g}(2 \alpha)$ is surjective. Take arbitrary $y \in \mathfrak{g}(2 \alpha)$. We define $x \in (\mathbb{C} z_0)^\perp$ by
\begin{equation*}
x := \frac{-1}{2 \|\alpha\|^2 \|z_0\|^2} [\bar{z_0}, y]
,\quad \text{where} \ \ \|z_0\|^2 := \langle z_0 , \bar{z}_0\rangle.
\end{equation*}
Then by the Jacobi identity we have
\begin{equation}\label{eq7.0.2}
\operatorname{ad}(z_0)[\bar{z}_0, y]
=
- [\bar{z}_0, [y, z_0]] - [y, [z_0, \bar{z}_0]]
=
[[z_0, \bar{z}_0], y],
\end{equation}
where the last equality follows from $[y, z_0] \in [\mathfrak{g}(2 \alpha) , \mathfrak{g}(\alpha)] \subset \mathfrak{g}(3 \alpha) = \{0\}$. Notice that $[z_0, \bar{z}_0] \in [\mathfrak{g}(\alpha), \mathfrak{g}(-\alpha)] \subset \mathfrak{g}(0)$. Moreover from \eqref{eq7.0} we have $[z_0, \bar{z}_0] \in \mathfrak{m}^\mathbb{C} \cap \mathfrak{p}^\mathbb{C}$. Hence we have $[z_0, \bar{z}_0] \in \mathfrak{t}^\mathbb{C}$ by maximality. Since 
\begin{equation*}
\langle [z_0, \bar{z_0}] , \eta \rangle
=
\langle \bar{z}_0 , [\eta, z_0] \rangle
=
\langle \bar{z}_0 , \sqrt{-1} \langle \alpha, \eta \rangle z_0 \rangle
=
\sqrt{-1} \|z_0\|^2\langle \alpha, \eta \rangle
\end{equation*}
for all $\eta \in \mathfrak{t}$ we get $[z_0, \bar{z}_0] = \sqrt{-1} \|z_0\|^2 \alpha$. Applying this to \eqref{eq7.0.2} we obtain
\begin{equation*}
\operatorname{ad}(z_0)[\bar{z}_0, y]
=
\sqrt{-1} \|z_0\|^2 [\alpha, y]
=
- 2 \|z_0\|^2  \| \alpha\|^2 y.
\end{equation*}
Therefore we have $\operatorname{ad}(z_0)(x) = y$.  This proves the lemma.
\end{proof}

Using this lemma we study the relation between the austere properties of $H$- and $P(G, H \times K)$-orbits in the rest of this section. First we consider the case $\sigma = \tau$ (Theorem II (i) in Introduction): 
\begin{thm}\label{prop1}
Let $M = G/K$ be a symmetric space of compact type and $H$ a symmetric subgroup of $G$. Suppose that $\sigma = \tau$. Then for $w \in \mathfrak{g}$ the following conditions are equivalent:
\begin{enumerate}
\item the orbit $H \cdot (\exp w)K$ through $(\exp w)K$ is an austere submanifold of $M$,
\item the orbit $P(G, H \times K) * \hat{w}$ through $\hat{w}$ is an austere PF submanifold of $V_\mathfrak{g}$.
\end{enumerate}
\end{thm}

\begin{rem}
The above conditions (i) and (ii) are also equivalent to the following conditions (see \cite[Proposition 4.27, Theorem 4.31]{Ika11} and \cite[Theorem 8]{M1}. See also \cite[Theorem 1]{M3} for the irreducible case):
\begin{enumerate}
\item[(iii)] the orbit $H \cdot (\exp w)K$ through $(\exp w)K$ is a totally geodesic submanifold of $M$,
\item[(iv)] the orbit $H \cdot (\exp w)K$ through $(\exp w)K$ is a reflective submanifold of $M$,
\item[(v)] the orbit $P(G, H \times K) * \hat{w}$ through $\hat{w}$ is a weakly reflective PF submanifold of $V_\mathfrak{g}$.
\end{enumerate}
Note that the orbit $P(G, H \times K) * \hat{w}$ is not totally geodesic (\cite[Corollary 2]{M1}) and thus not reflective. 
\end{rem}

\begin{proof}[Proof of Theorem \textup{\ref{prop1}}]
Take a maximal abelian subspace $\mathfrak{t}$ in $\mathfrak{m} = \mathfrak{p}$. Since $\pi(\exp \mathfrak{t})$ is a section of the $H$-action we can assume $w \in \mathfrak{t}$ without loss of generality. 

``(i) $\Rightarrow$ (ii)": Let $\alpha \in \Delta^+$ satisfy $\langle \alpha, w \rangle \notin \pi \mathbb{Z}$. Suppose that $\langle \alpha, w \rangle \notin \frac{\pi}{2} \mathbb{Z}$. Then from Lemma \ref{lem2} (i) there exists $\alpha' \in \Delta^+$ satisfying $\langle \alpha', w \rangle \notin \frac{\pi}{2} \mathbb{Z}$ such that
\begin{equation}\label{eq7.1}
\alpha \cot \langle \alpha, w \rangle = (-1) \times \alpha' \cot \langle \alpha', w \rangle.
\end{equation}
Since $\langle \alpha, w \rangle \notin \frac{\pi}{2}\mathbb{Z}$ we have $\alpha' \neq \alpha$. Then by the property of root systems $\alpha'$ is either $2 \alpha$ or $\frac{1}{2} \alpha$. Suppose that $\alpha' = 2 \alpha$. Then the multiplicities of left and right terms of \eqref{eq7.1} are $m(\alpha)$ and $m(2 \alpha)$ respectively. However  $m(\alpha) > m(2 \alpha)$ holds by Lemma \ref{lem3}. Thus $\alpha' \neq 2 \alpha$. Similarly $\alpha' \neq \frac{1}{2}\alpha$. This is a contradiction. Thus $\langle \alpha, w \rangle \in \frac{\pi}{2} \mathbb{Z}$ holds for all $\alpha \in \Delta^+$ satisfying $\langle \alpha, w  \rangle \notin \pi \mathbb{Z}$. (Thus $N$ is totally geodesic.) Hence from Lemma \ref{lem2} (ii) the orbit $P(G, H \times K) * \hat{w}$ is an austere PF submanifold of $V_\mathfrak{g}$.

``(ii) $\Rightarrow$ (i)":  Let $\alpha \in \Delta^+$ satisfy $\langle \alpha, w \rangle \notin \pi \mathbb{Z}$. Suppose that $\langle \alpha, w \rangle \notin \frac{\pi}{2} \mathbb{Z}$. Take $m \in \mathbb{Z}$. Then it follows from Lemma \ref{lem2} (ii) that there exist $\alpha' \in \Delta^+$ satisfying $\langle \alpha', w \rangle \notin \frac{\pi}{2} \mathbb{Z}$ and $m' \in \mathbb{Z}$ such that 
\begin{equation}\label{eq7.2}
\frac{1}{\langle \alpha, w \rangle + m \pi }\alpha
=
(-1) \times 
\frac{1}{ \langle \alpha', w \rangle + m' \pi }\alpha'.
\end{equation}
Since $\langle \alpha, w \rangle \notin \frac{\pi}{2}\mathbb{Z}$ we have $\alpha' \neq \alpha$. Then $\alpha'$ is either $2 \alpha$ or $\frac{1}{2} \alpha$. Suppose that $\alpha' = 2 \alpha$. Then the multiplicity of the left term is $m(\alpha) + m(2 \alpha)$ due to the equality $\frac{1}{\langle \alpha, w \rangle + m \pi}\alpha = \frac{1}{\langle 2 \alpha, w \rangle + 2 m \pi}2 \alpha$. However that of the right term is $m(2 \alpha)$ since $\alpha' \neq \alpha$. Thus we have $\alpha' \neq 2 \alpha$.  Similarly we have $\alpha' \neq \frac{1}{2}\alpha$. This is a contradiction. Thus $\langle \alpha, w  \rangle \in \frac{\pi}{2} \mathbb{Z}$ holds for all $\alpha \in \Delta^+$ satisfying $\langle \alpha, w \rangle \notin \pi \mathbb{Z}$.  This shows that the orbit $H \cdot (\exp w) K$ is totally geodesic and therefore austere. 
\end{proof}

To generalize Theorem \ref{prop1} we recall an equivalence relation for involutions: For two involutive automorphisms $\tau$ and $\tau'$ of $G$ we write $\tau \sim \tau'$ if there exists $c \in G$ such that $\tau' = \operatorname{Ad}(c) \circ \tau \circ  \operatorname{Ad}(c)^{-1}$. If $\tau \sim \tau'$ and $H$ a symmetric subgroup of $G$ with respect to $\tau$ then $H' := \operatorname{Ad}(c)H$ is a symmetric subgroup of $G$ with respect to $\tau'$. Moreover the actions of $H$ and $H'$ on $M$ are  conjugate, that is, there exists an isomorphism $\phi : H \rightarrow H'$ and an isometry $\psi: M \rightarrow M$ such that $\psi(b \cdot p) = \phi(b) \cdot \psi(p)$ for $b \in H$ and $p \in M$. In fact $\phi := \operatorname{Ad}(c)$ and $\psi := L_c$ satisfy the property. Thus we can identify $H'$-orbits with $H$-orbits via $\psi$ and the theorem is generalized as follows:

\begin{cor}\label{cor5}
Let $M$, $H$ be as in Theorem \textup{\ref{prop1}}. Suppose that $\sigma \sim \tau$. Then for $w \in \mathfrak{g}$ the following conditions are equivalent:
\begin{enumerate}
\item the orbit $H \cdot (\exp w )K$ through $(\exp w)K $ is an austere submanifold of $M$,
\item the orbit $P(G, H \times K) * \hat{w}$ through $\hat{w}$ is an austere PF submanifold of $V_\mathfrak{g}$.
\end{enumerate}
\end{cor}
\begin{proof}
Let $c \in G$ satisfy $\sigma = \operatorname{Ad}(c) \circ \tau \circ  \operatorname{Ad}(c)^{-1}$. 
Take $g \in P(G,G \times \{e\})$ satisfying $g(0) = c$. Then from \eqref{equiv} the diagram
\begin{equation*}
\begin{CD}
V_\mathfrak{g} @>g * >> V_\mathfrak{g}
\\
@V\Phi_K VV @V\Phi_K VV
\\
M @>L_c>> M
\end{CD}
\end{equation*}
commutes. Since each $P(G, H \times K)$-orbit is the inverse images of an $H$-orbit under $\Phi_K$ the assertion follows from Theorem \ref{prop1}.
\end{proof}

Next we consider the case $\sigma \circ \tau = \tau \circ \sigma$ (Theorem II (ii) in Introduction): 
\begin{thm}\label{thm:commute}
Let $M = G/K$ be a symmetric space of compact type and $H$ a symmetric subgroup of $G$. Suppose that the involutions $\sigma$ and $\tau$ commute. Then if the orbit $H \cdot (\exp w)K$ through $(\exp w) K$ where $w \in \mathfrak{g}$ is an austere submanifold of $M$, the orbit $P(G, H \times K) * \hat{w}$ through $\hat{w}$ is an austere PF submanifold of $V_\mathfrak{g}$.
\end{thm}

\begin{rem}
The converse of Theorem \ref{thm:commute} does not hold in general. In the next section we will show a counterexample of a minimal $H$-orbit which is \emph{not} austere but the corresponding minimal $P(G, H \times K)$-orbit is austere.
\end{rem}

\begin{proof}[Proof of Theorem \textup{\ref{thm:commute}}]
Take a maximal abelian subspace $\mathfrak{t}$ in $\mathfrak{m} \cap \mathfrak{p}$. We can assume $w \in \mathfrak{t}$ without loss of generality. Let $\alpha \in \Delta^+$. If the set $\mathbb{R} \alpha \cap \Delta^+$ consists of only $\alpha$ then it follows by the same argument as in the proof of Theorem \ref{main1} that the set
\begin{equation*}
\left\{
\left.
\frac{1}{\langle \alpha, w \rangle + \frac{1}{2} \arg \epsilon + m \pi} \alpha
\ \right|\ 
\epsilon \in U(1)_\alpha^*
, \ 
m \in \mathbb{Z}
\right\}
\end{equation*}
with multiplicities is invariant under the multiplication by $(-1)$. Let us consider the other cases $\mathbb{R} \alpha \cap \Delta^+ = \{\alpha, 2 \alpha\}$ or $\{\alpha, \frac{1}{2} \alpha\}$. It suffices to consider the former case. By Lemma \ref{lem2} the union $X \cup Y$ of two sets
\begin{align*}
& 
X := 
\left\{
\left.
\cot \left( \langle \alpha, w \rangle + \frac{1}{2} \arg \epsilon \right) \alpha
\ \right|\ 
\epsilon \in U(1)_\alpha^*
\right\} 
\quad \text{and}
\\
& 
Y := 
\left\{
\left.
\cot \left( \langle 2 \alpha, w \rangle + \frac{1}{2} \arg \delta \right) 2 \alpha
\ \right|\ 
\delta \in U(1)_{2\alpha}^*
\right\}
\end{align*}
with multiplicities is invariant under the multiplication by $(-1)$, and it suffices to show that the union $Z \cup W$ of two sets
\begin{align*}
& 
Z := 
\left\{
\left.
\frac{1}{\langle \alpha, w \rangle + \frac{1}{2} \arg \epsilon + m \pi} \alpha
\ \right|\ 
\epsilon \in U(1)_\alpha^*
, \ 
m \in \mathbb{Z}
\right\}
\quad  \text{and}
\\
& 
W :=
\left\{
\left.
\frac{1}{\langle 2\alpha, w \rangle + \frac{1}{2} \arg \delta + m \pi} 2 \alpha
\ \right|\ 
\delta \in U(1)_{2 \alpha}^*
, \ 
m \in \mathbb{Z}
\right\}
\end{align*}
with multiplicities is invariant under the multiplication by $(-1)$. 

Since $\sigma$ and $\tau$ commute we have $\epsilon, \delta \in \{\pm1\}$. Thus if $\langle \alpha , w \rangle \in \frac{\pi}{2} \mathbb{Z}$ then the sets $X$, $Y$, $Z$ and $W$ are empty. Suppose that $\langle \alpha , w \rangle \notin \frac{\pi}{2} \mathbb{Z}$. Then $\langle \alpha, w \rangle + \frac{1}{2} \arg \epsilon \notin \frac{\pi}{2} \mathbb{Z}$ for all $\epsilon \in U(1)_\alpha$. Thus $U(1)_\alpha^* = U(1)_\alpha$. Hence $m(\alpha) = \sum_{\epsilon \in U(1)_\alpha^*} m(\alpha, \epsilon)$. This implies that there exist $\epsilon, \epsilon' \in U(1)_{\alpha}^*$ such that 
\begin{equation}\label{7.2.1}
\cot \left( \langle \alpha, w \rangle + \frac{1}{2} \arg \epsilon \right) \alpha
=
(-1) \times \cot \left( \langle \alpha, w \rangle + \frac{1}{2} \arg \epsilon' \right) \alpha.
\end{equation}
In fact, if this does not hold then for each $\epsilon \in U(1)_{\alpha}^*$ there exists a unique $\delta(\epsilon) \in U(1)_{2 \alpha}^*$ satisfying
\begin{equation*}
\cot \left( \langle \alpha, w \rangle + \frac{1}{2} \arg \epsilon \right) \alpha
=
(-1) \times \cot \left( \langle2 \alpha, w \rangle + \frac{1}{2} \arg \delta(\epsilon) \right)2 \alpha.
\end{equation*}
The multiplicity of the left term is $m(\alpha, \epsilon)$, or $m(\alpha, \epsilon) + m(2 \alpha, \delta')$ if there exists $\delta' \in U(1)_{2 \alpha}^*$ satisfying $\cot ( \langle \alpha, w \rangle + \frac{1}{2} \arg \epsilon) \alpha = \cot ( \langle 2\alpha, w \rangle + \frac{1}{2} \arg \delta') 2 \alpha$. That of the right term is $m(2 \alpha, \delta(\epsilon))$; due to negation of \eqref{7.2.1} we have $\cot (\langle 2 \alpha, w  \rangle + \frac{1}{2} \arg \delta(\epsilon)) \neq \cot (\langle  \alpha, w  \rangle + \frac{1}{2} \arg \epsilon')$ for any $\epsilon' \in U(1)_\alpha^*$ and thus it is not $m(2 \alpha, \delta(\epsilon)) + m(\alpha, \epsilon')$ but $m(2 \alpha, \delta(\epsilon))$. Thus we get $m(\alpha, \epsilon) \leq m(2 \alpha, \delta(\epsilon))$. Hence we obtain
\begin{equation*}
m(\alpha) = \sum_{\epsilon \in U(1)_\alpha^*} m(\alpha, \epsilon) \leq \sum_{\epsilon \in U(1)_\alpha^*} m(2\alpha, \delta(\epsilon)) \leq m(2 \alpha),
\end{equation*}
where the last inequality is due to the injective property of the map $\epsilon \mapsto \delta(\epsilon)$. 
This contradicts the fact $m(\alpha) > m(2 \alpha)$ of Lemma \ref{lem3}. Thus from \eqref{7.2.1} we have
\begin{equation*}
\langle \alpha , w \rangle + \frac{1}{2} \arg \epsilon
=
(-1) \times 
\left(
\langle \alpha, w \rangle + \frac{1}{2} \arg \epsilon'
\right),
 \mod\pi\mathbb{Z}.
\end{equation*}
Thus $\langle \alpha, w \rangle = - \frac{1}{4} \arg \epsilon - \frac{1}{4} \arg \epsilon'$ mod $\pi \mathbb{Z}$. Since $\epsilon, \epsilon' \in \{\pm1\}$ we obtain $\langle \alpha, w \rangle \in \frac{\pi}{4} \mathbb{Z}$ and $\langle 2\alpha, w \rangle \in \frac{\pi}{2} \mathbb{Z}$. Thus
\begin{equation*}
\langle \alpha, w \rangle +  \frac{1}{2} \arg \epsilon 
\in 
\frac{\pi}{4} \mathbb{Z}
, \qquad
\langle 2\alpha, w \rangle +  \frac{1}{2} \arg \delta
\in 
\frac{\pi}{2} \mathbb{Z}.
\end{equation*}
for any $\epsilon \in U(1)_\alpha$ and $\delta \in U(1)_{2 \alpha}$. Thus the sets $Y$ and $W$ are empty. Hence the set $X$ with multiplicities is invariant under the multiplication by $(-1)$. Therefore by the same argument as in the proof of Theorem \ref{main1} the set $Z$ with multiplicities is invariant under the multiplication by $(-1)$. This proves the theorem.
\end{proof}

To generalize Theorem \ref{thm:commute} we recall an equivalence relation for pairs of involutions introduced by Matsuki \cite{Mat02}. Let $(\sigma, \tau)$ and $(\sigma', \tau')$ be two pairs of involutive automorphisms of $G$. We write $(\sigma, \tau) \sim (\sigma', \tau')$ if there exist an automorphism $\rho$ of $G$ and an element $c \in G$ such that 
\begin{equation*}
\sigma'  = \rho \circ \sigma \circ \rho^{-1}
, \qquad
\tau' = 
\operatorname{Ad}(c) \circ \rho \circ \tau \circ \rho^{-1} \circ \operatorname{Ad}(c)^{-1}.
\end{equation*}
If $K$ and $H$ are symmetric subgroups of $G$ then $K' :=  \rho (K)$ and $H' := \operatorname{Ad}(c) \circ \rho(H)$ are symmetric subgroups of $G$. Moreover the $H$-action on $G/K$ and the $H'$-action on $G/K'$ are conjugate. In fact these actions are conjugate under the isomorphism $\phi := \operatorname{Ad}(c) \circ \rho : H \rightarrow H'$ and the isometry $\psi: G/K \rightarrow G/K'$ defined by $\psi(aK) := c \rho(a)K'$. Then the theorem is generalized as follows:

\begin{cor}\label{cor:commute}
Let $M$, $H$ be as in Theorem \textup{\ref{thm:commute}}. Suppose that there exists a pair of commuting involutions $(\sigma', \tau')$ of $G$ satisfying $(\sigma, \tau) \sim (\sigma' , \tau')$. Then if the orbit $H \cdot (\exp w)K$ through $(\exp w ) K$ where $w \in \mathfrak{g}$ is an austere submanifold of $M$, the orbit $P(G, H \times K) * \hat{w}$ through $\hat{w}$ is an austere PF submanifold of $V_\mathfrak{g}$.
\end{cor}

\begin{proof}
Let $\rho \in \operatorname{Aut}(G)$ and $c \in G$ be as above. Since we equipped an $\operatorname{Aut}(G)$-invariant inner product with $\mathfrak{g}$ the automorphism $\rho$ is an isometry of $G$. Thus it induces an isometry from $G/K$ to $G/K'$, which is still denoted by $\rho$. The differential $d \rho : \mathfrak{g} \rightarrow \mathfrak{g}$ induces a linear orthogonal transformation of $V_\mathfrak{g}$ by pointwise operation, which is still denoted by $d \rho$. 
Note that $d \rho(g * \hat{0})  = (\rho \circ g) * \hat{0}$ holds for all $g \in \mathcal{G}$. Take $h \in P(G, G \times \{e\})$ satisfying $h(0) = c$. Then from \eqref{equiv} the diagram 
\begin{equation*}
\begin{CD}
V_\mathfrak{g} @>d \rho >> V_\mathfrak{g} @> h* >> V_\mathfrak{g}
\\
@V \Phi_K VV @V \Phi_{K'} VV @V \Phi_{K'} VV
\\
G/K @>\rho >>G / K'@>L_c>> G/K'
\end{CD}
\end{equation*}
commutes. Since each $P(G, H \times K)$-orbit is the inverse image of an $H$-orbit under $\Phi_K$ the assertion follows from Theorem \ref{thm:commute}.
\end{proof}

Finally, as far as possible, we consider the general case that $\sigma$ and $\tau$ do not necessarily commute. In view of Corollary \ref{cor:commute} it suffices to consider non-commutative pairs of involutions which are not equivalent to commutative ones. According to the classification result \cite{Mat02} if $G$ is simple then there are three kinds of such non-commutative pairs, and if $G$ is not simple then there are many such non-commutative pairs. For a technical reason, here we focus on the case that $G$ is simple. In this case if $(\sigma, \tau)$ is one of those three pairs then the order of the composition $\sigma \circ \tau$ is $3$ or $4$ (see also \cite[Section 5]{Ohno21}). We will use this fact to prove the following theorem (Theorem II (iii) in Introduction):

\begin{thm}\label{thm:main}
Let $M = G/K$ be a symmetric space of compact type and $H$ a symmetric subgroup of $G$. Suppose that $G$ is simple.  Then if the orbit $H \cdot (\exp w)K$ through $(\exp w)K$ where $w \in \mathfrak{g}$ is an austere submanifold of $M$, the orbit $P(G, H \times K) * \hat{w}$ through $\hat{w}$ is an austere PF submanifold of $V_\mathfrak{g}$.
\end{thm}

\begin{proof}
From the above discussion it suffices to consider a pair of involutions $(\sigma, \tau)$ where the order $l$ of $\sigma \circ \tau$ is $3$ or $4$. Take a maximal abelian subspace $\mathfrak{t}$ in $\mathfrak{m} \cap \mathfrak{p}$. We can assume $w \in \mathfrak{t}$ without loss of generality. Let $\alpha \in \Delta^+$. By the same argument as in the proof of  Theorem \ref{thm:commute} it suffices to consider the case $\mathbb{R} \alpha \cap \Delta^+ = \{\alpha, 2 \alpha\}$ and to show that the austere property of $X \cup Y$ implies that of $Z \cup W$.

First we show that $\langle \alpha, w \rangle \in \frac{\pi}{2l} \mathbb{Z}$. If $U(1)_\alpha^* \subsetneq U(1)_\alpha$ then there exists $\epsilon \in U(1)_\alpha$ satisfying $\langle \alpha, w \rangle + \frac{1}{2} \arg \epsilon \in \frac{\pi}{2}\mathbb{Z}$. This shows $\langle \alpha, w  \rangle \in \frac{\pi}{2l} \mathbb{Z}$. If $U(1)_\alpha^* = U(1)_\alpha$ then $m(\alpha) = \sum_{\epsilon \in U(1)_\alpha^*} m(\alpha, \epsilon)$. Thus by the same argument as in the proof of Theorem \ref{thm:commute}  there exist $\epsilon, \epsilon' \in U(1)_{\alpha}^*$ such that 
\begin{equation*}
\cot \left( \langle \alpha, w \rangle + \frac{1}{2} \arg \epsilon \right) \alpha
=
(-1) \times \cot \left( \langle \alpha, w \rangle + \frac{1}{2} \arg \epsilon' \right) \alpha.
\end{equation*}
From this we have 
\begin{equation*}
\langle \alpha , w \rangle + \frac{1}{2} \arg \epsilon
=
(-1) \times 
\left(
\langle \alpha, w \rangle + \frac{1}{2} \arg \epsilon'
\right),
\mod \pi \mathbb{Z}.
\end{equation*}
Hence $\langle \alpha, w \rangle = - \frac{1}{4} \arg \epsilon - \frac{1}{4} \arg \epsilon'$ mod $\pi \mathbb{Z}$. Therefore $\langle \alpha, w \rangle \in \frac{\pi}{2l} \mathbb{Z}$ as claimed.

Since $\langle \alpha, w \rangle \in \frac{\pi}{2l} \mathbb{Z}$ we have 
\begin{equation*}
\langle \alpha, w \rangle +  \frac{1}{2} \arg \epsilon 
\in 
\frac{\pi}{2l} \mathbb{Z}
, \qquad
\langle 2\alpha, w \rangle +  \frac{1}{2} \arg \delta
\in 
\frac{\pi}{l} \mathbb{Z}
\end{equation*}
for any $\epsilon \in U(1)_\alpha$ and $\delta \in U(1)_{2 \alpha}$. Thus if $l = 3$ then
\begin{equation*}
\cot \left(\langle \alpha, w \rangle + \frac{1}{2} \arg \epsilon \right) 
= 
\pm \sqrt{3}, \ \pm \frac{1}{\sqrt{3}}
, \qquad
\cot \left(\langle 2\alpha, w \rangle + \frac{1}{2} \arg \delta \right) 
= 
\pm \frac{1}{\sqrt{3}}
\end{equation*}
and if $l = 4$ then
\begin{equation*}
\cot \left(\langle \alpha, w \rangle + \frac{1}{2} \arg \epsilon \right) 
= 
\pm 1, \ \pm (\sqrt{2} \pm 1)
, \qquad
\cot \left(\langle 2\alpha, w \rangle + \frac{1}{2} \arg \delta \right) 
= 
\pm 1
\end{equation*}
for $\epsilon \in U(1)_\alpha^*$ and $\delta \in U(1)_{2\alpha}^*$. Therefore 
\begin{equation*}
\cot \left( \langle \alpha, w \rangle + \frac{1}{2} \arg \epsilon \right) \alpha
\neq
(-1) \times \cot \left( \langle 2\alpha, w \rangle + \frac{1}{2} \arg \delta \right) 2\alpha.
\end{equation*} 
for any $\epsilon \in U(1)_\alpha^*$ and $\delta \in U(1)_{2\alpha}^*$. This shows that the sets $X$ and $Y$ with multiplicities are respectively invariant under the multiplication by $(-1)$. Thus by the similar arguments as in the proof of Theorem \ref{main1} the sets $Z$ and $W$ with multiplicities are respectively invariant under the multiplication by $(-1)$. This proves the theorem.
\end{proof}

\begin{rem}
By the same arguments we can generalize Theorem \ref{thm:main} to the case that $G$ is not simple but the order of $\sigma \circ \tau$ is $3$ or $4$. 
\end{rem}

\begin{rem}
In the proofs of Theorems \ref{thm:commute} and \ref{thm:main} we essentially showed that the orbit $H \cdot (\exp w)$ is an austere submanifold of $M$ if and only if the set
\begin{equation*}
\left\{
\left.
\cot \left( \langle \alpha, w \rangle + \frac{1}{2} \arg \epsilon \right) \alpha
\ \right|\ 
\epsilon \in U(1)_\alpha^*
\right\}
\end{equation*}
with multiplicities is invariant under the multiplication by $(-1)$ for each $\alpha \in \Delta^+$. 
\end{rem}

\begin{example}
Ikawa \cite{Ika11} classified austere orbits of Hermann actions under the assumptions that $G$ is simple and that $\sigma$ and $\tau$ commute. This result was extended by Ohno \cite{Ohno21} to the non-commutative case. Thus applying  those results to Theorems \ref{thm:commute} and \ref{thm:main}  we can obtain many examples of homogeneous austere PF submanifolds in Hilbert spaces. Note that so obtained austere PF submanifolds are not totally geodesic due to Corollary 2 in \cite{M1}.
\end{example}

%%%%%%%
\section{A counterexample to the converse}\label{counterexample}

In this section we show a counterexample to the converse of Theorems \ref{thm:commute} and \ref{thm:main}; we show an example of a minimal $H$-orbit which is \emph{not} austere but the corresponding minimal $P(G, H \times K)$-orbit is austere. Note that from Theorem \ref{main1} the root system $\Delta$ must be non-reduced. We give such an example by the triple
\begin{equation*}
(G, K, H) = (SU(p + q), S(U(p) \times U(q)), SO(p + q)).
\end{equation*}
We shall suppose that $p>q$.

The involutions $\sigma$ and $\tau$ of $G$ corresponding to $K$ and $H$ respectively are 
\begin{equation*}
\sigma = \operatorname{Ad}(I_{pq})
\  \text{ where }\ 
I_{pq} = 
\left[\begin{array}{cc}
- E_p & 0 \\ 0 & E_q
\end{array}\right]
\ \text{ and } \ 
\tau : \text{complex conjugation},
\end{equation*}
where $E_p$ denote the unit matrix of order $p$. Clearly $\sigma$ and $\tau$ commute. The canonical decomposition of $\mathfrak{g} = \mathfrak{su}(p + q)$ with respect to $\sigma$ is given by $\mathfrak{k} = \mathfrak{s}(\mathfrak{u}(p) + \mathfrak{u}(q))$ and 
\begin{equation*}
\mathfrak{m}
=
\left\{\left.
	\left[
	\begin{array}{ccc}
	0&0& Z
	\\
	0&0&W
	\\
	-{}^t \bar{Z}&-{}^t \bar{W}&0
	\end{array}
	\right]
\ \right|\ 
Z \in \mathfrak{gl}(q, \mathbb{C})
, \ 
W \in \mathfrak{gl}(p-q, q, \mathbb{C})
\right\}.
\end{equation*}
The canonical decomposition of $\mathfrak{g}$ with respect to $\tau$ is given by $\mathfrak{h} = \mathfrak{so}(p + q)$ and 
\begin{equation*}
\mathfrak{p}
=
\{\sqrt{-1} X  \mid X \in \operatorname{Sym}(p+q, \mathbb{R}), \  \operatorname{tr} X = 0\}.
\end{equation*}
Thus we can write
\begin{equation*}
\mathfrak{m} \cap \mathfrak{p}
=
\left\{\left.
\sqrt{-1}
	\left[
	\begin{array}{ccc}
	0&0& X
	\\
	0&0&Y
	\\
	{}^t X&{}^t Y&0
	\end{array}
	\right]
\ \right|\ 
X \in \mathfrak{gl}(q, \mathbb{R})
, \ 
Y \in \mathfrak{gl}(p-q, q, \mathbb{R})
\right\}.
\end{equation*}
We define a maximal abelian subspace $\mathfrak{t}$ in $\mathfrak{m} \cap \mathfrak{p}$ by
\begin{equation*}
\mathfrak{t}
=
\left\{
\left.
\sqrt{-1}
	\left[
	\begin{array}{ccc}
	0&0& X
	\\
	0&0&0
	\\
	X&0&0
	\end{array}
	\right]
\ \right|\ 
X =
\left[\begin{array}{ccc}
x_1 \vspace{-1mm}&& \\ &\!\ddots\!& \vspace{-1mm}\\ && x_q
\end{array}\right]
, \ 
x_1, \cdots , x_q \in \mathbb{R}
\right\}.
\end{equation*}
Note that $\mathfrak{t}$ is maximal also in $\mathfrak{m}$. For each $i = 1, \cdots , q$ we set
\begin{equation*}
e_i
=
\sqrt{-1}
	\left[
	\begin{array}{ccc}
	0&0& E_{ii}
	\\
	0&0&0
	\\
	E_{ii}&0&0
	\end{array}
	\right].
\end{equation*}
where $E_{ij}$ denote the square matrix of order $q$ having $1$ in the $i$-th row and $j$-th column and zeros elsewhere. We set
\begin{align*}
&
\mathfrak{m}_{2e_i}
=
\left\{
\left.
	\left[
	\begin{array}{ccc}
	0&0& X^{(i)}
	\\
	0&0&0
	\\
	-{}^t X^{(i)}&0&0
	\end{array}
	\right]
\ \right|\ 
X^{(i)} =
x E_{ii}
, \ 
x \in \mathbb{R}
\right\},
\\
&
\mathfrak{m}_{e_i}
=
\left\{\left.
\left[
	\begin{array}{ccc}
	0&0& 0
	\\
	0&0&W^{(i)}
	\\
	0&-{}^t \bar{W}^{(i)}&0
	\end{array}
\right]
\ \right| \ 
W^{(i)} = 
\left[\begin{array}{ccc}
0 & w_{1, i}&0
\\
0& \vdots& 0
\\
0&w_{p-q, i}&0
\end{array}\right]
, \ 
w_{1,i}, \cdots , w_{p-q, i} \in \mathbb{C}
\right\},
\\
&
\mathfrak{m}_{e_i \pm e_j}
=
\left\{\left.
\left[
	\begin{array}{ccc}
	0&0& Z^{(i,j)}
	\\
	0&0& 0
	\\
	-{}^t \bar{Z}^{(i,j)}&0&0
	\end{array}
\right]
\ \right| \ 
Z^{(i,j)} = zE_{ij} \mp \bar{z}E_{ji}
, \ 
z \in \mathbb{C}
\right\},
\end{align*}
where
\begin{equation*}
\dim \mathfrak{m}_{2e_i} = 1
, \qquad
\dim \mathfrak{m}_{e_i} = 2 (p-q)
, \qquad
\dim \mathfrak{m}_{e_k + e_l} = \dim \mathfrak{m}_{e_k - e_l} =  2.
\end{equation*}
Then we obtain the root space decomposition
\begin{equation*}
\mathfrak{m} 
= 
\mathfrak{t} 
+ 
\sum_{i = 1}^q \mathfrak{m}_{2 e_i}
+ 
\sum_{i = 1}^q \mathfrak{m}_{e_i} 
+ 
\sum_{1 \leq i < j \leq q} \mathfrak{m}_{e_i + e_j}
+
\sum_{1 \leq i < j \leq q} \mathfrak{m}_{e_i - e_j}.
\end{equation*}
By commutativity of involutions this decomposition is refined as follows:
\begin{align*}
&
\mathfrak{m} \cap \mathfrak{p} 
=  
\mathfrak{t}  
+ 
\sum_{i =1}^q  (\mathfrak{m}_{e_i} \cap \mathfrak{p})
+ 
\sum_{1 \leq i < j \leq q} (\mathfrak{m}_{e_i + e_j} \cap \mathfrak{p})
+
\sum_{1 \leq i < j \leq q} (\mathfrak{m}_{e_i - e_j} \cap \mathfrak{p}).
\\
&
\mathfrak{m} \cap \mathfrak{h} 
= 
\sum_{i = 1}^q \mathfrak{m}_{2 e_i} 
+
\sum_{i =1}^q  (\mathfrak{m}_{e_i} \cap \mathfrak{h}) 
+ 
\sum_{1 \leq i < j \leq q} (\mathfrak{m}_{e_i + e_j} \cap \mathfrak{h})
+ 
\sum_{1 \leq i < j \leq q} (\mathfrak{m}_{e_i - e_j} \cap \mathfrak{h}),
\end{align*}

We now consider the orbit $N := H \cdot (\exp w) K$, where $w \in \mathfrak{t}$ is defined by
\begin{equation*}
w := \frac{\pi}{8} \sum_{i = 1}^q e_i.
\end{equation*}
Set $a := \exp w$. Then from \eqref{tangent3:commute} and \eqref{normal3:commute} the tangent space and the normal space of $N$ are
\begin{align*}
T_{aK} N 
&=
dL_a ( \ 
	\sum_{i =1}^q  (\mathfrak{m}_{e_i} \cap \mathfrak{p}) + \sum_{1 \leq i < j \leq q} (\mathfrak{m}_{e_i + e_j} \cap \mathfrak{p}) \ )
\\
& + dL_a( \ \sum_{i = 1}^q \mathfrak{m}_{2 e_i}   + \sum_{i =1}^q  (\mathfrak{m}_{e_i} \cap \mathfrak{h}) + \sum_{1 \leq i < j \leq q} (\mathfrak{m}_{e_i + e_j} \cap \mathfrak{h}) +  \sum_{1 \leq i < j \leq q} (\mathfrak{m}_{e_i - e_j} \cap \mathfrak{h})\ ),
\\
T^\perp_{aK} N &= dL_a(\ \mathfrak{t} + \sum_{1 \leq i < j \leq q} (\mathfrak{m}_{e_i - e_j} \cap \mathfrak{p})\ ),
\end{align*}
and for each $\xi \in \mathfrak{t}$ the principal curvatures of $N$ in the direction of $dL_a(\xi)$ are expressed as the inner product of $\xi$ with vectors 
\begin{equation*}
-(\sqrt{2} + 1)e_i
, \quad
-(e_i + e_j)
, \quad 
2 e_i 
, \quad 
(\sqrt{2} - 1)e_i 
, \quad 
e_i + e_j
, \quad 
0,
\end{equation*}
whose multiplicities are respectively
\begin{equation*}
p - q 
, \quad
1
, \quad
1
, \quad 
p-q
, \quad
1
, \quad
1.
\end{equation*}
Since the set $\{- (\sqrt{2} + 1)e_i, 2 e_i, (\sqrt{2} - 1)e_i\}$ can not be invariant under the multiplication by $(-1)$ the orbit $N$ is not an austere submanifold of $M$. Note that if $p-q = 1$ then it is a minimal submanifold of $M$ but still not austere.

On the other hand, from Corollary \ref{cor1} (see also Remark \ref{rem5}) the principal curvatures of the orbit $P(G, H \times K) * \hat{w}$ in the direction of $\hat{\xi}$ are expressed as the inner product of $\xi$ with vectors
\begin{align*}
&
\{0\}
, \ \ 
\left\{
-\frac{1}{\frac{\pi}{8} + m \pi} e_i
\right\}_{m \in \mathbb{Z}}
, \ \ 
\left\{
- \frac{1}{\frac{\pi}{4} + m \pi} (e_i + e_j)
\right\}_{m \in \mathbb{Z}},
\\
&
\left\{
- \frac{1}{\frac{3}{4}\pi+ m \pi} 2 e_i
\right\}_{m \in \mathbb{Z}}
, \ \ 
\left\{
- \frac{1}{\frac{5}{8}\pi+ m \pi}  e_i
\right\}_{m \in \mathbb{Z}}
, \ \ 
\left\{
- \frac{1}{\frac{3}{4}\pi + m \pi}  (e_i + e_j)
\right\}_{m \in \mathbb{Z}},
\\
&
\left\{
- \frac{1}{\frac{\pi}{2}+ m \pi}  (e_i - e_j)
\right\}_{m \in \mathbb{Z}}
, \ \ 
\left\{
\frac{1}{n \pi}  (e_i - e_j)
\right\}_{n \in \mathbb{Z} \backslash \{0\}},
\end{align*}
whose multiplicities are respectively
\begin{equation*}
\infty
, \quad 
p-q
, \quad
1
, \quad
1
, \quad
p-q
, \quad
1
, \quad
1
, \quad
1.
\end{equation*}
Note that
\begin{equation*}
\left\{
- \frac{1}{\frac{3}{4}\pi+ m \pi} 2 e_i
\right\}_{m \in \mathbb{Z}}
=
\left\{
- \frac{1}{\frac{3}{8}\pi+ m \pi}  e_i
\right\}_{m \in \mathbb{Z}}
\cup
\left\{
- \frac{1}{ \frac{7}{8}\pi + m \pi}  e_i
\right\}_{m \in \mathbb{Z}}.
\end{equation*}
Note also that the sets
$\{- \frac{1}{\pi/2+ m \pi}  (e_i - e_j)\}_{m \in \mathbb{Z}}$
and
$\{\frac{1}{n \pi}  (e_i - e_j)\}_{n \in \mathbb{Z} \backslash \{0\}}$
with multiplicities are respectively invariant under the multiplication by $(-1)$. Thus from the equalities
\begin{equation*}
\frac{1}{\frac{\pi}{8} + m \pi} e_i = (-1) \times \frac{1}{\frac{7}{8}\pi + (- m-1)\pi} e_i
, \quad
\frac{1}{\frac{5}{8}\pi + m \pi} e_i = (-1) \times \frac{1}{\frac{3}{8}\pi + (-m-1)\pi} e_i,
\end{equation*}
\begin{equation*}
\frac{1}{\frac{\pi}{4} + m \pi} (e_i + e_j) = (-1) \times \frac{1}{\frac{3}{4}\pi + (-m-1)\pi} (e_i + e_j)
\end{equation*}
the orbit $P(G, H \times K)* \hat{w}$ is austere if and only if $p - q = 1$. Therefore we have shown that if $p - q = 1$ then the orbit $H \cdot (\exp w) K$ is not austere but the orbit $P(G, H \times K) * \hat{w}$ is austere. This is the desired counterexample. In this case the orbit $H \cdot (\exp w) K$ is a minimal submanifold of $M$ as mentioned above, and thus the orbit $P(G, H \times K) * \hat{w}$ is minimal PF submanifold of $V_\mathfrak{g}$ (\cite{KT93}, \cite{HLO06}). 

Finally we mention further remarks on the converse. As we have seen above, if the root system $\Delta$ is non-reduced then there exists a counterexample to the converse of Theorems \ref{thm:commute} and \ref{thm:main}. However, even if $\Delta$ is non-reduced, the converse holds in some cases. In fact, Theorem \ref{prop1} and Corollary \ref{cor5} are valid in the non-reduced case. Moreover consider the case $\sigma \circ \tau = \tau \circ \sigma$ and set
\begin{equation*}
\Delta_1^+ := \{\alpha \in \Delta^+ \mid \mathfrak{m}_\alpha \cap \mathfrak{p} \neq \{0\}\}
\quad \text{and} \quad
\Delta_{-1}^+ := \{\alpha \in \Delta^+ \mid \mathfrak{m}_\alpha \cap \mathfrak{h} \neq \{0\}\}.
\end{equation*}
Suppose that $\Delta$ is of type $BC$ and write $\Delta^+ = \{e_i,  2 e_i\}_i \cup \{e_i \pm e_j\}_{i<j}$.
Suppose also that $\dim \mathfrak{t} \geq 2$ and $\Delta_1^+ \cap \Delta_{-1}^+ = \{ e_i\}_i$. Then it follows by straightforward calculations that the converse holds (cf.\ \cite{M5}). Note that the counterexample shown in this section satisfies $\Delta_1^+ \cap \Delta_{-1}^+ = \{e_i\}_i \cup \{e_i \pm e_j\}_{i < j}$ if $\dim \mathfrak{t} \geq 2$, and $\Delta_1^+ \cap \Delta_{-1}^+ = \{e_1\}$ if $\dim \mathfrak{t} = 1$. For the investigation of the triple $(\Delta, \Delta_1, \Delta_{-1})$ and the corresponding commutative Hermann actions, see Ikawa's papers \cite{Ika11} and \cite{Ika14}.

%%%%%%%%%%
\subsection*{Acknowledgements}
The author would like to thank Professor Yoshihiro Ohnita for useful discussions and valuable suggestions. The author is also grateful to Professors Shinji Ohno, Hiroyuki Tasaki and Hiroshi Tamaru for their interests in this work and useful comments. Thanks are also due to Professors Naoyuki Koike and Takashi Sakai for their encouragements.

\end{document}